\documentclass[paper=A4,twoside=true,parskip=half,toc=bib,headinclude=true,titlepage=true,numbers=noenddot,cleardoublepage=plain,11pt,DIV=12,BCOR=10mm]{scrartcl}
\pdfoutput=1

\usepackage{amsmath}
\usepackage{amssymb}
\usepackage{amsthm}
\usepackage{graphics}
\usepackage{tabularx}

\usepackage[pdftex,pdfpagelabels,colorlinks=true,linkcolor=blue,citecolor=red]{hyperref}

\usepackage[american]{babel}
\usepackage[automark]{scrpage2}

\clearscrheadfoot
\ohead{\headmark}
\ofoot[\pagemark]{\pagemark}
\addtokomafont{pageheadfoot}{\small}

\newtheoremstyle{mytheorem}
  {\baselineskip}
  {\baselineskip}
  {\itshape}
  {}
  {\bfseries}
  {.}
  {.5em}
  {}
\newtheoremstyle{myremark}
  {\baselineskip}
  {\baselineskip}
  {}
  {}
  {\bfseries}
  {.}
  {.5em}
  {}

\theoremstyle{mytheorem}

\numberwithin{equation}{section}
\newtheorem{thm}{Theorem}
\newtheorem{prop}{Proposition}[section]
\newtheorem{lem}[prop]{Lemma}
\newtheorem{cor}[prop]{Corollary}
\newtheorem{thmapx}[prop]{Theorem}

\theoremstyle{myremark}

\newtheorem{rem}[prop]{Remark}

\newcommand{\assumptions}{(H1)--(H2) }
\newcommand{\dashlist}{--}

\newcommand{\openI}{I}
\newcommand{\IR}{\mathbb{R}}
\newcommand{\IN}{\mathbb{N}}

\newcommand{\possubset}{U}
\newcommand{\possubsett}{U'}
\newcommand{\posparam}{\sigma}

\newcommand{\proj}{P}
\newcommand{\discspace}{V}
\newcommand{\basisfcn}{e}
\newcommand{\functional}{x'}

\newcommand{\mobextra}{m_0}
\newcommand{\gk}{N}
\newcommand{\ddt}{\dfrac{d}{dt}}
\newcommand{\ddx}{\partial_x}

\newcommand{\inttime}[1]{\int\limits_0^T {#1} \, dt}
\newcommand{\inttimeex}[3]{\int\limits_{#1}^{#2} {#3} \, dt}
\newcommand{\intspace}[1]{\int\limits_\Omega {#1} \, dx}
\newcommand{\intspaceex}[2]{\int\limits_{#1} {#2} \, dx}
\newcommand{\intspacetime}[1]{\iint\limits_{\Omega_T} {#1} \, dx \, dt}
\newcommand{\intspacetimeex}[2]{\iint\limits_{#1} {#2} \, dx \, dt}
\newcommand{\intex}[4]{\int\limits_{#1}^{#2}{#3} \, #4}

\newcommand{\possettimeabs}[1]{\left\lbrace \abs{#1} > 0 \right\rbrace_T}
\newcommand{\possettime}[1]{\left\lbrace #1 > 0 \right\rbrace_T}
\newcommand{\posset}[1]{\left\lbrace #1 > 0 \right\rbrace}
\newcommand{\possetex}[2]{\left\lbrace #1 > #2 \right\rbrace}
\newcommand{\zeroset}[1]{\left\lbrace #1 = 0 \right\rbrace}
\newcommand{\setex}[1]{\lbrace #1 \rbrace}

\newcommand{\prodl}[2]{\bigl( #1 , #2 \bigr)_{L^2(\Omega)}}
\newcommand{\pairh}[2]{\bigl\langle #1 , #2 \bigr\rangle_{H^1(\Omega)}}
\newcommand{\prodex}[3]{\bigl( #1 , #2 \bigr)_{#3}}
\newcommand{\pairex}[3]{\bigl\langle #1 , #2 \bigr\rangle_{#3}}
\newcommand{\norml}[1]{\bigl\lVert #1 \bigr\rVert_{L^2(\Omega)}}
\newcommand{\norm}[2] {\bigl\lVert #1 \bigr\rVert_{#2}}
\newcommand{\bochner}[2]{#1 ( \openI ; #2 (\Omega) )}
\newcommand{\bochnerex}[3]{#1 ( #2 ; #3 )}
\newcommand{\evspace}{W^1_2 \left( \openI ; H^1(\Omega) ; L^2(\Omega) \right)}
\newcommand{\timespace}{\bochnerex{L^2}{I}{H^1(\Omega)'}}

\newcommand{\weakto}{\rightharpoonup}
\newcommand{\strongto}{\rightarrow}
\newcommand{\weakstarto}{\stackrel{\ast}{\rightharpoonup}}
\newcommand{\weaktext}{\, \text{ in }}
\newcommand{\strongtext}{\, \text{ in }}
\newcommand{\weakstartext}{\, \text{ in }}
\newcommand{\pwaetext}{\, \text{ pointwise almost everywhere in }}

\newcommand{\embedding}{\hookrightarrow}
\newcommand{\cptembedding}{\stackrel{c}{\hookrightarrow}}

\providecommand{\abs}[1]{\lvert #1 \rvert}
\providecommand{\absB}[1]{\Bigl\vert #1 \Bigr\vert}

\newcommand{\Q}[1]{\sqrt{1+\lvert #1 \rvert^2}}
\newcommand{\Qex}[1]{\dfrac{#1}{\sqrt{1+\lvert #1 \rvert^2}}}

\newcommand{\cma}{,}
\newcommand{\fst}{.}
\newcommand{\beginstep}[2]{\textit{#2}}
\newcommand{\leftA}{\biggl(}
\newcommand{\rightA}{\biggr)}

\newcommand{\approxu}{u_0}

\newcommand{\qedthm}{\hfill \(\blacksquare\)}
\newcommand{\qedrem}{\hfill \(\square\)}

\newcommand{\limeps}{\to}
\newcommand{\sequence}[2]{\lbrace #1 \rbrace_{#2}}

\DeclareMathOperator{\supp}{supp}
\DeclareMathOperator*{\esssup}{ess\,sup}

\DeclareMathOperator{\divergence}{div}

\newcommand{\clearsection}{\clearpage}

\begin{document}

\allowdisplaybreaks[2]
\setcounter{page}{-1}
\setcounter{section}{-1}

\begin{titlepage}
\begin{center}
\large
\vspace*{0.5cm}

Diploma Thesis\\
\vspace*{1.65cm}
{\Huge\bfseries{On the Thin-Film Equation with\\[0.5ex]Nonlinear Surface Tension Term}\\}
\vspace*{1.65cm}

{\Large Jan Friederich\\}
\vspace*{1.65cm}

Erlangen,\\
June 8, 2009\\
\vspace*{3.2cm}

Supervisor: Prof. Dr. G\"{u}nther Gr\"{u}n\\[2ex]
Chair of Applied Mathematics I\\
Department of Mathematics \\
Friedrich-Alexander-Universit\"{a}ät Erlangen-N\"{u}rnberg\\
\vspace*{0.5cm}
\includegraphics{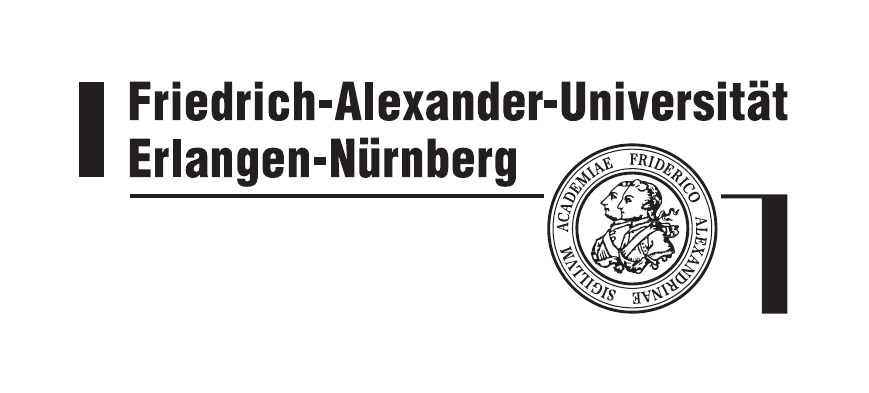}

\end{center}
\end{titlepage}

\cleardoubleoddemptypage

\tableofcontents
\cleardoublepage

\pagestyle{scrheadings}

\section{Introduction}

In this thesis, we are concerned with existence and nonnegativity results for solutions of the fourth-order nonlinear degenerate parabolic partial differential equation
\begin{equation} \label{P1}
\begin{aligned}
u_t &- \divergence \left( m(u) \, \nabla p \right) = 0 \\
p &= - \divergence \Qex{\nabla u} 
\end{aligned}
\qquad \text{ in } \IR^+ \times \Omega \cma
\end{equation}
where \(\Omega \subset \IR^1\) is an open, bounded interval. Equation \eqref{P1} is supplemented by Neumann boundary conditions for \(u\) and \(p\). For the nonlinear mobility \(m(.)\) we have in mind functions of the form \(m(u) = \abs{u}^n\), \(n \geq 1\). Problem \eqref{P1} is closely related to the thin-film equation
\begin{equation} \label{P2}
\begin{aligned}
u_t &- \divergence \left( m(u) \, \nabla p \right) = 0 \\
p &= - \triangle u 
\end{aligned}
\qquad \text{ in } \IR^+ \times \Omega \cma
\end{equation}
where the curvature term
has been replaced by the linear approximation \(-\triangle u\).


Serving as a model problem for the evolution of thin liquid films, equation \eqref{P2} has been extensively studied throughout the last two decades. The unknown \(u = u(t,x)\) describes the thickness of a thin film of viscous fluid that spreads over a horizontal solid surface. The motion of the fluid is solely driven by surface tension, which is modeled by the pressure term \(p = p(t,x)\). In this regard, the definition of the nonlinear function \(m(.)\) is determined by the flow condition that is imposed at the liquid-solid surface. For instance, for a no-slip condition one considers \(m(u) = \abs{u}^3\), whereas various slip conditions entail mobility functions of the form \(m(u) = \abs{u}^3 + \beta \abs{u}^s\), \(0 < s < 3\), \(\beta > 0\). For a derivation of the model problem \eqref{P2}, which is based on lubrication approximation for the incompressible Navier--Stokes equations, we refer the reader to Oron, Davis and Bankoff \cite{oron}.

The mathematical investigation of problem \eqref{P2} started in 1990 with the paper of Bernis and Friedman \cite{berfri}, who were the first to prove existence of a weak solution to \eqref{P2} in one space dimension, supplemented by Neumann boundary conditions \(u_x = p_x = 0\). In addition, Bernis and Friedman were able to show that initially nonnegative solutions stay nonnegative for all times.  This is a remarkable result in comparison to non-degenerate fourth-order equations, solutions of which may in general take negative values even if the initial data are strictly positive. Their study of equation \eqref{P2} relies on two kinds of estimates: The energy estimate, which reads in one space dimension as
\begin{equation} \label{int0energy2}
\intspace{\abs{u_x(T, x)}^2} + \intspacetime{m(u) \, \abs{u_{xxx}}^2} = \intspace{\abs{u_{0,x}}^2} \cma
\end{equation}
and the entropy estimate
\begin{equation} \label{int0entropy2}
\intspace{ \int\limits_a^{u(T, x)}{\int\limits_a^s{\frac{1}{m(r)} \, dr} \, ds}} + \intspacetime{\abs{u_{xx}}^2} = \intspace{ \int\limits_a^{u_0(x)}{\int\limits_a^s{\frac{1}{m(r)} \, dr} \, ds}} \fst
\end{equation}
As another noteworthy feature of equation \eqref{P2}, one observes that the evolution of a thin liquid film may in fact be considered as a free-boundary problem. More precisely, the free boundary is given by the contact line of fluid, solid and air, that is \(\partial \posset{u}\). In this context, it turns out that the qualitative behaviour of the solution \(u\) depends on the growth exponent \(n\) of the mobility \(m(.)\) at zero. In one space dimension, Beretta, Bertsch and Dal Passo \cite{berberdal} showed that the support of the solution \(u\) is non-decreasing if \(n \geq \frac{3}{2}\) and stays constant in time if \(n \geq 4\). Moreover, the solution \(u\) exhibits the property of finite speed of propagation if \(0 < n < 3\) (see Bernis \cite{bernis1,bernis2} and Hulshof and Shishkov \cite{hulshi}). These results have partly been extended to higher space dimensions. For instance, we refer to Gr\"{u}n \cite{gruen1} and Elliott and Garcke \cite{ellgar} for existence results and to Beretta, Dal Passo, Garcke and Gr\"{u}n \cite{berdalgargruen} and Dal Passo, Gracke and Gr\"{u}n \cite{dalgargruen} for results on the qualitative behaviour of solutions. We also mention the occurrence of waiting time phenomena in multiple space dimensions (see Dal Passo, Giacomelli and Gr\"{u}n \cite{dalgiagruen}).


Recently, Becker \cite{becker} introduced a slight modification of the lubrication approximation appearing in \cite{oron} to derive an evolution equation for the film height. In the course of this derivation, the linear approximation to the surface tension \(-\triangle u\) is replaced by the exact curvature term \(- \nabla \cdot (\nabla u / \Q{\nabla u})\), which finally leads to problem \eqref{P1}. Related to this equation, the analogous energy estimate in one space dimension can be formally derived as
\begin{equation} \label{int0energy1}
\intspace{\Q{u_x(T, x)}} + \intspacetime{m(u) \, \abs{p_x}^2} = \intspace{\Q{\smash[b]{u_{0,x}}}} \cma
\end{equation}
whereas the entropy estimate reads as
\begin{equation} \label{int0entropy1}
\intspace{ \int\limits_a^{u(T, x)}{\int\limits_a^s{\frac{1}{m(r)} \, dr} \, ds}} + \intspacetime{\frac{\abs{u_{xx}}^2}{\Q{u_x}^3}} = \intspace{ \int\limits_a^{u_0(x)}{\int\limits_a^s{\frac{1}{m(r)} \, dr} \, ds}} \fst
\end{equation}
In view of the surface energy terms appearing in the corresponding estimates \eqref{int0energy1} and \eqref{int0energy2} respectively, one expects solutions of problem \eqref{P1} to exhibit spherical-cap-shaped profiles, whereas solutions of the standard thin-film equation \eqref{P2} are in general paraboloid-shaped. This has indeed been confirmed by the numerical simulations presented by Becker \cite{becker}. In addition, Becker proved existence and nonnegativity of discrete solutions to problem \eqref{P1} in multiple space dimensions.

It is the aim of the present thesis to prove existence and nonnegativity of weak solutions to problem \eqref{P1} in the continuous setting in one space dimension. Following the arguments of \cite{berfri,gruen1,ansgia}, our investigation relies solely on the energy estimate \eqref{int0energy1} and the entropy estimate \eqref{int0entropy1}. However, in comparison to the standard thin-film equation \eqref{P2} and its associated estimates \eqref{int0energy2} and \eqref{int0entropy2}, we will have to deal with two essential analytical difficulties. First, the pressure \(p\) is given by a nonlinear second-order differential operator applied to \(u\), which is the so-called minimal surface operator. Secondly -- and even more important --, this operator is not coercive, which
leads to very weak a priori estimates. Contrary to equation \eqref{P2}, where the estimates \eqref{int0energy2} and \eqref{int0entropy2} immediately imply that \(u\) is bounded in \(\bochnerex{L^\infty}{\IR^+}{H^1(\Omega)} \cap \bochnerex{L^2_{loc}}{\IR^+}{H^2(\Omega)}\), it is not obvious at all whether the estimates \eqref{int0energy1} and \eqref{int0entropy1} allow to infer any reasonable a priori estimate for weak derivatives of \(u\) in any reflexive function space.

However, we suggest an argument that combines the estimates \eqref{int0energy1} and \eqref{int0entropy1} and yields an estimate for \(u\) in the space \(H^2(\Omega)\) for almost all times (see Lemma \ref{lemcrazy}).
This allows to identify the pressure \(p\) and, in turn, to prove existence of a solution in a certain weak sense. In addition, we may partly employ of the arguments of Bernis and Friedman \cite{berfri} to prove that the solution \(u\) is nonnegative or even positive for almost all times depending on the growth exponent \(n \geq 1\).

Although our overall strategy might be adopted for a proof of existence in multiple space dimensions, some of the crucial arguments (including the aforementioned Lemma \ref{lemcrazy}) are valid in one space dimension only. Hence, we will confine ourselves to the one-dimensional setting in the sequel. In particular, an existence result in higher space dimensions as well as qualitative results analogous to those mentioned above remain open.

Let us briefly describe the organization of this thesis. In section \ref{secpreliminaries}, after introducing some notation, the problem and the results are stated and an outline of the proofs is given. Using an multi-step approximation procedure, sections \ref{secapproximate}, \ref{secdegenerate} and \ref{secfinal} are devoted to the proofs of the theorems. Finally, in section \ref{secperspectives} we shall discuss several open problems related to equation \eqref{P1}.

\clearsection

\clearsection

\section{Preliminaries} \label{secpreliminaries}

\subsection{Notation and Facts}

In the following, we will consider functions on the open, bounded intervals \(\Omega = (-l, l)\) and \(I = (0,T)\), where \(l, T > 0\) are arbitrary positive numbers. We set \( \Omega_T := (0,T) \times \Omega \).

We will use subscripts to denote partial derivatives, e.g. \(u_x = \ddx u\). Superscripts will mostly denote sequence indices or parameters and must not be confused with exponents, e.g. \(\sequence{u^\epsilon}{\epsilon > 0}\). We will drop arguments of functions if there is no misunderstanding possible. Generic positive constants are denoted by \(C(.)\), where the arguments indicate the dependencies on parameters (except on \(T\) and \(\Omega\)). The value of a constant may change within an estimate.

For a Lebesgue-measurable set \(A\), \(\mu(A)\) denotes the Lebesgue measure and \(\chi_A\) the characteristic function of \(A\). For functions \(u : \Omega_T \to \IR \) and \(\rho \geq 0\) we define
\begin{align*}
\possetex{u}{\rho}_T &:= \lbrace (t,x) \in \Omega_T : u(t, x) > \rho \rbrace \subseteq \Omega_T \cma \\[5pt]
\possetex{u(t)}{\rho} &:= \lbrace x \in \Omega : u(t, x) > \rho \rbrace \subseteq \Omega \cma
\end{align*}
and analogously for \(<, = , \geq, \leq\).

Let \(X\) be a real Banach space with norm \(\norm{.}{X}\). Then, \(X'\) denotes the dual of \(X\) and \[\pairex{.}{.}{X} := \pairex{.}{.}{X' \times X}\] the pairing of elements of \(X'\) and \(X\). If \(X\) is a Hilbert space, \[\prodex{.}{.}{X} := \prodex{.}{.}{X \times X}\] denotes the scalar product. We write \(X \embedding Y\) and \(X \cptembedding Y\) for the continuous and compact embedding, respectively, of \(X\) into another Banach space \(Y\). \(\mathcal{L}(X;Y)\) is the space of continuous linear operators mapping from \(X\) to \(Y\). We use the symbols \(\strongto\), \(\weakto\) and \(\weakstarto\) to denote strong, weak and weak-* convergence, respectively.

We will deal with the following spaces of functions on \(\Omega\):

\begin{itemize}
\item By \(L^p(\Omega)\), \(1 \leq p \leq \infty\), we denote the spaces of measurable functions which are integrable to the \(p\)-th power. These are Banach spaces equipped with the usual norms
\[
\norm{f}{L^p(\Omega)} := \Bigl( \intspace{\abs{f(x)}^p} \Bigr)^\frac{1}{p} \quad (1 \leq p < \infty) \cma \qquad \norm{f}{L^\infty(\Omega)} := \esssup\limits_{x\in\Omega} \, \abs{f(x)} \fst
\]
\(L^p(\Omega)\) is separable if \(1 \leq p < \infty\) and reflexive if \(1 < p < \infty\). \(L^2(\Omega)\) is a Hilbert space, where the scalar product is given by
\[
\prodl{f}{g} := \intspace{f(x) \, g(x)} \fst
\]

\item The Sobolev spaces \(W^{k,p}(\Omega)\), \(k \in \IN\), \(1 \leq p \leq \infty\), are the spaces of weakly differentiable functions in \(L^p(\Omega)\) whose distributional derivatives up to order \(k\) are elements of \(L^p(\Omega)\) as well. These are Banach spaces equipped with the norms
\[
\begin{split}
\norm{f}{W^{k,p}(\Omega)} &:= \Bigl( \sum\limits_{0\leq j \leq k} \norm{\partial_j f}{L^p(\Omega)}^p \Bigr)^\frac{1}{p} \quad (1 \leq p < \infty) \cma \\
\norm{f}{W^{k,\infty}(\Omega)} &:= \max\limits_{0\leq j \leq k} \norm{\partial_j f}{L^\infty(\Omega)} \fst
\end{split}
\]
\(W^{k,p}(\Omega)\) is separable if \(1 \leq p < \infty\) and reflexive if \(1 < p < \infty\).

\item We set \( H^k(\Omega) := W^{k,2}(\Omega) \), \(k \in \IN\). 

\item \(H^1_0(\Omega)\) denotes the Banach space comprising functions of \(H^1(\Omega)\) that attain zero boundary values in the sense of traces.

\item \(L^p_{loc}(\Omega)\) (\(W^{k,p}_{loc}(\Omega)\), \(k \in \IN\)), \(1 \leq p \leq \infty\), are the spaces of locally and integrable (weakly differentiable) functions. That is, \(f \in L^p_{loc}(\Omega)\) (\(f \in W^{k,p}_{loc}(\Omega)\)) holds if \(f \in L^p(U)\) (\(f \in W^{k,p}(U)\)) for all \(U \subset \subset \Omega\).

\item \(BV(\Omega)\) is the space of functions of bounded variation on \(\Omega\). We will barely make use of the properties of \(BV(\Omega)\) and hence refer to \cite{evans} for a precise definition.

\item \(C^{k,\alpha}(\overline{\Omega})\), \(k \in \IN\), \(0 \leq \alpha \leq 1\), are the spaces of H\"{o}lder-continuous functions on \(\overline{\Omega}\) equipped with the usual norms (see for instance \cite{adams}).

\item \(C^\infty_0(\Omega)\) denotes the set of smooth functions that have compact support in \(\Omega\).
\end{itemize}

Moreover, we will deal with spaces of vector-valued functions on the interval \(I = (0,T)\). In the following, let \(X\) be a Banach space.
\begin{itemize}
\item \(\bochnerex{L^p}{\openI}{X}\), \(1 \leq p \leq \infty\), is the space of \(X\)-valued Bochner-measurable functions whose integrals exist up to the \(p\)-th power. \(\bochnerex{L^p}{\openI}{X}\) are Banach spaces equipped with the norms
\[
\begin{split}
\norm{f}{\bochnerex{L^p}{\openI}{X}} &:= \Bigl( \intex{\openI}{}{\norm{f(t)}{X}^p}{dt} \Bigr)^\frac{1}{p} \quad (1 \leq p < \infty) \cma \\
\norm{f}{\bochnerex{L^\infty}{\openI}{X}} &:= \esssup\limits_{t \in I} \, \norm{f(t)}{X} \fst \\
\end{split}
\]
\(\bochnerex{L^p}{\openI}{X}\) is separable if \(1 \leq p < \infty\) and \(X\) is separable. \(\bochnerex{L^p}{\openI}{X}\) is reflexive if \(1 < p < \infty\) and \(X\) is reflexive.

\item By \(\bochnerex{C}{[0,T]}{X}\) we denote the space of continuous functions with values in \(X\). The norm is given by
\[
\norm{f}{\bochnerex{C}{[0,T]}{X}} := \sup\limits_{t \in [0,T]} \, \norm{f(t)}{X} \fst
\]

\item The spaces \(\bochnerex{C^k}{[0,T]}{X}\), \(k \in \IN\), are defined in an analogous way, where the derivatives up to order \(k\) are understood in the sense of Fr\'{e}chet (cf. \cite{zeidler1}).
\end{itemize}

We will often make use the following facts (cf. for instance \cite{emmrich,ruz}):
\begin{itemize}
\item \(\bochnerex{L^p}{\openI}{X'} \cong \bochnerex{L^q}{\openI}{X}'\) for \(1 < p, q < \infty\), \(\frac{1}{p} + \frac{1}{q} = 1\).
\item \(\bochnerex{L^p}{\openI}{L^p(\Omega)} \cong L^p(\Omega_T)\) for \(1 \leq p < \infty\).
\item \(\bochnerex{L^\infty}{\openI}{L^\infty(\Omega)} \subseteq L^\infty(\Omega_T)\).
\end{itemize}

Some more important facts involving vector-valued functions are cited in the Appendix. Note in particular the results on so-called ``evolution triples'' \ref{propevtrip} and Simon's compactness criterion \ref{thmsimon}, which will be of major importance in the sequel.

We assume the reader to be familiar with some well-known theorems and results from calculus and functional analysis, e.g. Vitali's convergence theorem, Fatou's lemma, the theorem of Arzel\`{a}--Ascoli, H\"{o}lder's and Young's inequality, as well as embedding results for Sobolev spaces, which we will frequently use in the course of the proofs.
For reference we suggest \cite{adams,alt,zeidler1,zeidler2}.

\subsection{Statement of the Theorems}

In this thesis, we consider the following system of partial differential equations:
\begin{equation} \label{chap0sys}
\left\lbrace
\begin{aligned}
u_t &- \partial_x \left( m(u) \, p_x \right) = 0 && \text{ in } (0,T) \times \Omega \\
p &= - \partial_x \Qex{u_x} && \text{ in } (0,T) \times \Omega \\
u_x &= p_x = 0 && \text{ on } (0,T) \times \partial \Omega \\
u &= u_0 && \text{ on } \{t = 0\} \times \Omega
\end{aligned}
\right.
\end{equation}
Here, \(\Omega = (-l, l)\), \(l > 0\), is an open, bounded interval. In addition, we formulate the following hypotheses:

\begin{enumerate}
\item[(H1)] The mobility \(m: \IR \to \IR^+_0\) is a continuous nonnegative function given by
\begin{equation}
 m(s) = \abs{s}^n \, \mobextra(s) \qquad \forall \, s \in \IR
\end{equation}
with growth exponent \(n \geq 1\) and a smooth function \(\mobextra \in C^1(\IR)\) satisfying
\begin{equation} \label{chap0mbounds}
\mobextra(s) \geq c_{\mobextra} > 0 \qquad \forall \, s \in \IR
\end{equation}
for a positive constant \(c_{\mobextra}\).

\item[(H2)] For the initial data \(u_0\) we assume that
\begin{align}
u_0 &\in H^1(\Omega) \label{chap0u0h1} \\
\intertext{and}
u_0 &\geq 0 \label{chap0u0pos}
\end{align}
on \(\overline{\Omega}\). Furthermore, we require that
\begin{equation} \label{chap0boundGu0}
\intspace{G(u_0)} < \infty,
\end{equation}
where the function \(G : \IR \to \IR^+\) is given by (cf. section \ref{chapentropy} below)
\[
G(t) = \int\limits_t^a{\int\limits_s^a{\frac{1}{m(r)} \,dr} \, ds}, \qquad a > \sup\limits_{x \in \overline{\Omega}} u_0(x) \fst
\]
\end{enumerate}

Note that condition \eqref{chap0boundGu0} is quite restrictive. In fact, it is easy to show that \eqref{chap0u0h1}\dashlist\eqref{chap0boundGu0} imply that \(\mu \left(\zeroset{u_0} \right) = 0 \) if \(n \geq 2\), and that \(u_0 > 0\) on \(\overline{\Omega}\) if \(n \geq 4\) (cf. \cite[Remark 4.1]{berfri} and the proof of Proposition \ref{proposition3pos} below). Hence, \(u_0\) may have compact support only in the case \(1 \leq n < 2\).

The aim of this work is to prove the following theorems. The first one yields existence of a weak solution in the following sense:

\begin{thm} \label{theorem1}
Assume (H1)--(H2). Then a pair of functions \((u, p)\) exists such that
\begin{align}
u &\in \bochner{L^\infty}{BV} \cap \bochnerex{C}{[0,T]}{L^q(\Omega)}, \quad 1 \leq q < \infty, \label{thm1regu} \\
u_t &\in \timespace, \label{thm1regut} \\
p &\in L^2(\Omega_T), \label{thm1regp} \\
u(0) &= u_0 \label{thm1initial}, \\
\intspace{u(t)} &= \intspace{u_0}, \quad t \in [0,T], \label{thm1consmass}
\end{align}
and a set \(S \subset (0,T)\) with \(\mu(S) = T\) such that
\begin{align}
u(t) &\in H^2(\Omega), \label{thm1regu2} \\
u(t) &\geq 0 \text{ on } \overline{\Omega}, \label{thm1nonneg} \\
u_x(t) &= 0 \text{ on } \partial \Omega, \label{thm1boundary} \\
p_x(t) &\in L^2_{loc}(\posset{u(t)}) \label{thm1regpx}
\end{align}
for all \(t \in S\). Furthermore, we have
\[
m(u) \, p_x \, \chi_{\possettime{u}} \in L^2(\Omega_T)
\]
and the following equations are satisfied:
\begin{alignat}{2}
\inttime{\pairh{u_t}{v}} &+ \intspacetimeex{\possettime{u}}{m(u) \, p_x \cdot v_x} = 0 &&\qquad \forall \, v \in \bochner{L^2}{H^1} \label{thm1eqa} \\
p(t) &= - \partial_x \frac{u_x(t)}{\Q{u_x(t)}} &&\qquad \forall \, t \in S \label{thm1eqb}
\end{alignat}
\qedthm
\end{thm}

\begin{rem}
The regularity results of Theorem \ref{theorem1} are quite weak and may be interpreted as follows. We have \(u(t) \in H^2(\Omega)\) for almost all \(t \in (0,T)\). By standard embedding results for the one-dimensional setting, this implies that \(u(t) \in C^{1,\frac{1}{2}}(\overline{\Omega})\) for almost all \(t \in (0,T)\). That is, the function \(u(t)\) is indeed quite smooth in space, if we fix \(t \in S\). In particular, the results of Theorem \ref{theorem1} imply that a zero contact angle is attained at \(\partial \posset{u(t)}\) for almost all times \(t \in (0,T)\). However, the numbers \(\norm{u(t)}{H^2(\Omega)}\) might increase rapidly in time such that we do not even manage to prove that \(u \in \bochner{L^1}{H^2}\). \qedrem
\end{rem}

\begin{rem}
For comparison, in the setting of the standard thin-film equation in one space dimension (i.e. \(p = - u_{xx}\)) one typically obtains at least the following regularity results for the solution \(u\):
\begin{align*}
u &\in \bochner{L^2}{H^2} \cap \bochner{L^\infty}{H^1} \cap \bochnerex{C}{[0,T]}{L^q(\Omega)}, \quad 1 \leq q < \infty, \\
u_{xxx} &\in L^2_{loc}(\posset{u}_T)
\intertext{In particular, one is able to prove that}
u &\in C^{\frac{1}{8}, \frac{1}{2}}_{t,x}(\overline{\Omega}_T) \cma
\end{align*}
that is, \(u\) is globally H\"{o}lder-continuous on \(\overline{\Omega}_T\) (cf. for instance \cite{berfri}). \qedrem
\end{rem}

Note that the solution \(u\) obtained by Theorem \ref{theorem1} is nonnegative for almost all times. We may prove the following additional nonnegativity and positivity results, depending on the growth exponent \(n\) of the mobility \(m(.)\):

\begin{thm} \label{theorem2}
Assume (H1)--(H2) and let \((u, p)\) be a solution in the sense of Theorem \ref{theorem1}. In particular, let \(u(t) \geq 0\) on \(\overline{\Omega}\) for all \(t \in S\), where \(\mu(S) = T\). Then the following hold:
\begin{enumerate}
\item[(i)] If \(n \geq 2\), then \(\mu \left( \zeroset{u(t)} \right) = 0 \) for all \(t \in S\).
\item[(ii)] If \(n \geq \frac{8}{3}\), then \(u(t) > 0\) on \(\overline{\Omega}\) for all \(t \in S\).\qedthm
\end{enumerate}
\end{thm}

\begin{rem}
Note that Bernis and Friedman proved the following positivity result for solutions of the standard thin-film equation (see \cite[Theorem 4.1]{berfri}):

{\centering \textit{If \(n \geq 4\), then \(u > 0\) holds on \(\overline{\Omega}_T\).}

}
This result is by no means weaker than the second statement of Theorem \ref{theorem2}, which holds for \(n \geq \frac{8}{3}\). On the contrary, solutions \(u\) having property Theorem \ref{theorem2} (ii) may in fact be \(0\) on a set of measure zero in \(\Omega_T\). Moreover, Bernis and Friedman prove the very same result as Theorem \ref{theorem2} (ii) in the setting of the standard thin-film equation (see \cite[Corollary 4.5]{berfri}). \qedrem
\end{rem}

\subsection{Plan of the Proofs}

We will prove Theorem \ref{theorem1} by a multi-step approximation procedure. That is, we introduce several regularizations, depending on real numbers \(0 < \delta, \epsilon, \eta \leq 1\), in order to deal with the difficulties that arise from the degeneracy of the equation and the non-coercive structure of the pressure \(p\) in \eqref{chap0sys}. The regularized problem we will consider is the following:
\begin{equation} \label{chap0sysreg}
\left\lbrace
\begin{aligned}
u_t &- \partial_x \left( m_{\epsilon,\eta}(u) \, p_x \right) = 0 &&\text{in } (0,T) \times \Omega \\
p &= - \partial_x \leftA \Qex{u_x} + \delta u_x \rightA &&\text{in } (0,T) \times \Omega \\
u_x &= p_x = 0 &&\text{on } (0,T) \times \partial \Omega \\
u &= \approxu &&\text{on } \{t=0\} \times \Omega
\end{aligned}
\right.
\end{equation}
Here, the mobility \(m\) has been replaced by a bounded and non-degenerate approximation \(m_{\epsilon,\eta}\) satisfying 
\begin{equation*}
\epsilon \leq m_{\epsilon,\eta}(.) \leq \frac{1}{\eta} + 1 \fst
\end{equation*}

Our studies of system \eqref{chap0sysreg} will heavily rely on the following energy estimate:
\begin{equation} \label{chap0energy}
\sup\limits_{t \in [0,T]} \intspace{\Q{u_x(t)}} + \frac{\delta}{2} \norml{u_x(t)}^2 + \intspacetime{ m_{\epsilon,\eta}(u) \, \abs{p_x}^2 } \leq C(\approxu)
\end{equation}
In addition, one easily observes that solutions \(u\) of equation \eqref{chap0sysreg} conserve mass.

In section \ref{chapgalerkin}, we will prove existence of a weak solution of system \eqref{chap0sysreg} by means of the Faedo--Galerkin method. The boundedness of \(m_{\epsilon,\eta}\) provided by \(\eta > 0\) is a necessary ingredient in this context. However, note that \eqref{chap0energy} and the fact that \(W^{1,1}(\Omega) \embedding L^\infty(\Omega)\) -- which holds true in one space dimension only -- implies that any solution of \eqref{chap0sysreg} is uniformly bounded independently of \(\delta, \epsilon, \eta > 0\). Hence, we may easily let \(\eta \to 0\) to prove existence in the case of unbounded mobilities afterwards. This will be done in section \ref{chapunbounded}. In these and the following limit processes, the pressure \(p\) can be identified by straightforward arguments relying on the relation
\[
-\ddx \Qex{u_x} = - \frac{u_{xx}}{\Q{u_x}^3} \cma
\]
which holds true in one space dimension. Nevertheless, we will briefly outline an alternative way of identifying the nonlinearity. This method is based on the theory of monotone operators and would work in multiple space dimensions.

So far, the positivity of \(m_\epsilon(.) = m(.) + \epsilon\) allowed to control the pressure \(p\) and its gradient. In section \ref{secdegenerate}, we follow the lines of Bernis and Friedman \cite{berfri} to prove that the approximating solutions \(\sequence{u^\epsilon}{\epsilon > 0}\) are indeed H\"{o}lder-continuous on \(\overline{\Omega}_T\) independently of \(\epsilon > 0\) (but depending on \(\delta > 0\)). Therefore, a subsequence of \(\sequence{u^\epsilon}{\epsilon > 0}\) converges uniformly to a continuous function \(u\) as \(\epsilon \to 0\). This will help to control the gradient of the pressure \(p\) on the set \(\posset{\abs{u}}_T\) as \(\epsilon \to 0\). Moreover, we will prove that the so-called entropy estimate
\begin{equation} \label{chap0entropy}
\sup\limits_{t \in [0,T]} \intspace{G_\epsilon(u(t))} + \intspacetime{\dfrac{\abs{u_{xx}}^2}{\Q{u_x}^3} + \delta  \abs{u_{xx}}^2} \leq C(u_0)
\end{equation}
holds, where \(G_\epsilon: \IR \to \IR^+\) is a positive primitive of \(m_\epsilon\). Using this estimate, we will be able to control second-order derivatives of \(u\) and identify the pressure \(p\) in the limit. After letting \(\epsilon \to 0\), we will prove nonnegativity of the solution \(u\) by the arguments of Bernis and Friedman \cite{berfri}, relying on the boundedness of \(G_\epsilon(u)\) and the continuity of \(u\).

In section \ref{secfinal}, we will finally prove Theorems \ref{theorem1} and \ref{theorem2} by letting \(\delta \to 0\). In view of \eqref{chap0energy} and \eqref{chap0entropy}, note that the quality of the a priori estimates substantially changes when \(\delta \to 0\). In particular, the energy estimate \eqref{chap0energy} is not strong enough to infer any compactness results for the gradient of \(u\), and it is not clear whether the entropy estimate \eqref{chap0entropy} can help to improve this situation. Furthermore, results on non-parametric mean curvature type equations are not applicable either due to the higher-order structure of problem \eqref{chap0sysreg}. However, relying on both estimates \eqref{chap0energy} and \eqref{chap0entropy}, we will show that the approximating solutions \(\sequence{u^\delta(t)}{\delta > 0}\) are bounded in \(H^2(\Omega)\) for almost all times \(t \in (0,T)\) independently of \(\delta > 0\) (see Lemma \ref{lemcrazy}). This will be the most crucial step in the proof of Theorem \ref{theorem1} as we will henceforth be able to identify the pressure \(p\) for almost all \(t \in (0,T)\). Moreover, we will immediately obtain that \(\sequence{u^\delta(t)}{\delta > 0}\) converges to a continuous function \(u(t)\) for almost all \(t \in (0,T)\) as \(\delta \to 0\) and hence argue similarly as before to identify the flux \(m(u(t)) \, p_x(t)\) on the positivity set \(\posset{u(t)}\). This will prove Theorem \ref{theorem1}. For the proof of the additional nonnegativity and positivity results of Theorem \ref{theorem2}, we will basically repeat the arguments from the previous section.




\clearsection

\section{Approximate Solutions} \label{secapproximate}

In the first step, we will apply the Faedo--Galerkin method to prove existence of a weak solution \((u,p)\) of a regularized version of system \eqref{chap0sys}. We consider the system
\begin{equation} \label{chap1sys}
\left\lbrace
\begin{aligned}
u_t &- \partial_x \left( m_{\epsilon,\eta}(u) \, p_x \right) = 0 &&\text{in } (0,T) \times \Omega \cma \\
p &= - \partial_x \leftA \Qex{u_x} + \delta u_x \rightA &&\text{in } (0,T) \times \Omega \cma \\
u_x &= p_x = 0 &&\text{on } (0,T) \times \partial \Omega \cma \\
u &= \approxu &&\text{on } \{t=0\} \times\Omega \cma
\end{aligned}
\right.
\end{equation}
where \(0 < \delta, \epsilon, \eta \leq 1\) are arbitrary. Here, the regularized mobility \(m_{\epsilon,\eta} \) is given by
\begin{equation}
m_{\epsilon,\eta}(s) := \dfrac{m(s)}{1 + \eta \, m(s)} + \epsilon \qquad \forall \, s \in \IR \fst \label{chap1mreg}
\end{equation}
Note in particular that
\begin{align}
&m_{\epsilon,\eta}(s) \leq \dfrac{1}{\eta} + 1 \cma \label {chap1mregbound} \\ 
&m_{\epsilon,\eta}(s) \geq \epsilon \label {chap1mregpos}
\end{align}
hold for \(0 < \epsilon, \eta \leq 1\) and arbitrary \(s \in \IR\).

Moreover, for \(0 < \delta \leq 1\) we define the nonlinear operator \(A_\delta : H^1(\Omega) \to H^1(\Omega)'\) by
\begin{equation}
\pairh{A_\delta(u)}{v} = \intspace{\leftA \Qex{u_x} + \delta u_x \rightA \cdot v_x } \qquad \forall \, u, v \in H^1(\Omega) \fst \label{chap1Adeltadef}
\end{equation}
We will make use of the following, rather rough bounds on \(A_\delta\), which hold obviously true for \(0 < \delta \leq 1\) and arbitrary \(u, v, w \in H^1(\Omega)\):
\begin{align}
\pairh{A_\delta(u)}{v} &\leq (1 + \delta) \, \norm{u}{H^1(\Omega)} \norm{v}{H^1(\Omega)} \label{chap1Adeltabound} \\
\pairh{A_\delta(w)}{w} &\geq \delta \, \norm{w_x}{L^2(\Omega)}^2 \label{chap1Adeltacoerv} 
\end{align}
In view of \eqref{chap1Adeltadef} and \eqref{chap1Adeltabound}, the nonlinear operator \(A_\delta : H^1(\Omega) \to H^1(\Omega)'\) induces an operator \(\tilde{A}_\delta : \bochner{L^2}{H^1} \to \bochnerex{L^2}{I}{H^1(\Omega)'}\)
given by \(\bigl(\tilde{A}_\delta (u)\bigr)(t) = A_\delta\bigl(u(t)\bigr)\). Using this relation, we will identify \(A_\delta\) and \(\tilde{A}_\delta\) in the sequel.

Furthermore, one observes that the operator \(A_\delta\) is monotone and hemicontinuous for \(\delta \geq 0\). This provides us with an elegant method of identifying \(A_\delta(u)\) and in turn the pressure \(p\) in the limit. However, it will turn out that \(p\) can be identified even more easily in our setting by straightforward arguments such that we will not need to make use of the theory of monotone operators throughout this work. In this context, also confer Remark \ref{remarkmon} below. Nevertheless, we will use the definition of \(A_\delta\) \eqref{chap1Adeltadef} for shorthand notation and in order to apply the estimates \eqref{chap1Adeltabound} and \eqref{chap1Adeltacoerv} directly.

\subsection{Galerkin Approximation} \label{chapgalerkin}

Our existence result for problem \eqref{chap1sys} is stated in the following proposition:

\begin{prop}\label{proposition1}
Assume \assumptions and let \(0 < \delta, \epsilon, \eta \leq 1\). Then a pair of functions \((u, p) = (u^{\delta, \epsilon, \eta}, p^{\delta, \epsilon, \eta})\) exists such that
\begin{align}
u &\in \bochner{L^2}{H^2} \cap \bochner{L^\infty}{H^1} \cap \bochnerex{C}{[0,T]}{L^2(\Omega)}, \label{prop1regu} \\
u_t &\in \timespace, \label{prop1regut} \\
p &\in \bochner{L^2}{H^1}, \label{prop1regp} \\
u_x(t) &\in H^1_0(\Omega) \text{ a.e. in } (0,T) \cma \label{prop1boundary} \\
u(0) &= \approxu \label{prop1initial},
\end{align}
which solves \eqref{chap1sys} in the following sense:
\begin{subequations} \label{prop1sys}
\begin{align}
\inttime{\pairh{u_t}{v}} &+ \intspacetime{m_{\epsilon,\eta}(u) \; p_x \cdot v_x} = 0 &\forall& \, v \in \bochner{L^2}{H^1} \label{prop1sysa} \\
p &= - \ddx \leftA \Qex{u_x} + \delta u_x \rightA \label{prop1sysb}
\end{align}
\end{subequations}

Furthermore, the solution satisfies the energy estimate
\begin{equation} \label{prop1energy}
\begin{split}
\intspace{\Q{u_x(t)}} + \frac{\delta}{2} \norml{u_x(t)}^2 + \intspacetime{ m_{\epsilon,\eta}(u_x) \, \abs{p_x}^2 } \leq C(\approxu)
\end{split}
\end{equation}
for almost all \(t \in (0,T)\). \qedthm
\end{prop}

Following the lines of \cite{gruen2}, we intend to prove Proposition \ref{proposition1} by means of the Faedo--Galerkin method. That is, we discretize equations \eqref{chap1sys} in space and solve for \(\gk \in \IN\) an \((\gk+1)\)-dimensional system of first-order ordinary differential equations. Afterwards, we will let \(\gk \to \infty\) and obtain a solution in the sense of Proposition \ref{proposition1}.

As ansatz functions we choose the \(L^2(\Omega)\)-normalized eigenfunctions of the Laplacian subject to Neumann boundary conditions, denoted by \(\sequence{\basisfcn_j}{j \in \IN}\). That is, the functions \(\sequence{\basisfcn_j}{j \in \IN}\) satisfy
\begin{equation*}
\left\lbrace
\begin{aligned}
-\basisfcn_{j,xx} &= \lambda_j \basisfcn_j &&\text{ in } \Omega \cma \\
\basisfcn_{j,x} &= 0 &&\text{ on } \partial \Omega \cma \\
\norml{\basisfcn_j} &= 1 \fst &&\\
\end{aligned}
\right.
\end{equation*}
where \(0 < \lambda_1 < \lambda_2 < \dotsc\) are the positive eigenvalues. Of course, for \(\Omega = (-l,l) \subset \IR^1\) we immediately obtain
\[
\basisfcn_j (x) = \frac{1}{\sqrt{l}} \cdot \cos \left( \sqrt{\lambda_j} \, x + \frac{\pi}{2} j \right) \cma \qquad \lambda_j = \left( \frac{\pi}{2l} j \right)^2 \cma \qquad j \in \IN \fst
\]
In addition we set
\[
\basisfcn_0 \equiv \mu(\Omega)^{-1/2} \equiv \frac{1}{\sqrt{2l}} \fst
\]
By the spectral theorem (see for instance \cite{alt}), the set \(\sequence{\basisfcn_j}{j \in \IN_0}\) is an orthonormal basis of \(L^2(\Omega)\). Moreover, the functions \(\sequence{\basisfcn_j}{j \in \IN_0}\) are obviously orthogonal in \(H^1(\Omega)\) and we observe that \(\sequence{\basisfcn_j}{j \in \IN_0} \subseteq C^\infty(\overline{\Omega}) \).


For \(\gk \in \IN\) we define
\[
\discspace_\gk := \operatorname{span} \lbrace\basisfcn_j : 0 \leq j \leq \gk \rbrace 
\]
and consider the orthogonal projection \(\proj_\gk : L^2(\Omega) \to  \discspace_\gk\) given by
\[
\proj_\gk v := \sum\limits_{i=0}^\gk \prodl{v}{\basisfcn_i} \, \basisfcn_i \qquad \forall \, v \in L^2(\Omega) \fst
\]
Finally, employing the \(H^1(\Omega)\)-orthogonality of the functions \(\sequence{\basisfcn_j}{j \in \IN_0}\) we easily deduce that
\begin{equation} \label{chap1boundproj}
\proj_\gk \in \mathcal{L}\left( H^1(\Omega); H^1(\Omega) \right) \cma \qquad \norm{\proj_\gk}{\mathcal{L}\left( H^1(\Omega); H^1(\Omega) \right)} = 1 \fst
\end{equation}

%
%


Now, for fixed \(\gk \in \IN\) we consider the Galerkin approximations
\begin{subequations} \label{chap1discreteup}
\begin{align}
u^\gk(t,x) &= \sum\limits_{i=0}^\gk{ c^\gk_i(t) \, \basisfcn_i(x)} \cma \label{chap1discreteu} \\
p^\gk(t,x) &= \sum\limits_{i=0}^\gk{ d^\gk_i(t) \, \basisfcn_i(x)} \cma \label{chap1discretep}
\end{align}
\end{subequations}
where the functions \(c^\gk_j, d^\gk_j : [0,T] \to \IR\) have to be determined from the following system of ordinary differential equations:
\begin{subequations} \label{chap1discretesys}
\begin{align}
\prodl{\ddt u^\gk(t)}{\basisfcn_j} &+ \prodl{m_{\epsilon,\eta}(u^\gk(t)) \, p_x^\gk(t)}{\basisfcn_{j,x}} = 0 \cma & j &= 0,1,\dotsc,\gk \label{chap1discretesysa} \\
\prodl{p^\gk(t)}{\basisfcn_k} &= \pairh{A_\delta (u^\gk(t))}{\basisfcn_k} \cma & k &= 0,1,\dotsc,\gk \label{chap1discretesysb} \\
u^\gk(0) &= u^N_0 := \proj_\gk \approxu \label{chap1discretesysc}
\end{align}
\end{subequations}

We summarize our existence result for system \eqref{chap1discretesys} in the following Lemma:

\begin{lem} \label{lemdisc}
Assume \assumptions and let \(0 < \delta, \epsilon, \eta \leq 1\). Then, for \(\gk \in \IN\) there exists a pair of functions \((u^N,p^N)\) such that \(u^N\) and \(p^N\) have the form \eqref{chap1discreteup} and satisfy equations \eqref{chap1discretesys}.

Furthermore, the energy estimate
\begin{equation} \label{lemdiscenergy}
\intex{\Omega}{ }{\Q{u^\gk_x(t)}}{dx} + \frac{\delta}{2} \norml{u^\gk_x(t)}^2 + \intspacetime{ m_{\epsilon,\eta}(u^\gk_x(s)) \, \abs{p^\gk_x(s)}^2 } \leq C(\approxu)
\end{equation}
holds for \(t \in [0,T]\). \qedthm
\end{lem}

\begin{proof}
Inserting \eqref{chap1discreteup} into \eqref{chap1discretesys}, we may employ the orthogonality of the functions \(\sequence{\basisfcn_j}{j \in \IN_0}\) to obtain by at first formal calculation
\begin{subequations} \label{prfdiscsys1}
\begin{align}
\ddt c^\gk_j(t) &+ \sum\limits_{i=0}^\gk{d^\gk_i(t) \prodl{m_{\epsilon,\eta}(u^\gk(t)) \, \basisfcn_{i,x}}{\basisfcn_{j,x}}} = 0 \cma & j &= 0,1,\dotsc,\gk \cma  \label{prfdiscsys1a} \\
d^\gk_k(t) &= \pairh{A_\delta (u^\gk(t))}{\basisfcn_k} \cma & k &= 0,1,\dotsc,\gk \cma \label{prfdiscsys1b} \\
c^\gk_l(0) &= \prodl{\approxu}{\basisfcn_l} \cma & l &= 0,1,\dotsc,\gk \fst \label{prfdiscsys1c}
\end{align}
\end{subequations}
Eliminating \(d^\gk_k(t)\) from \eqref{prfdiscsys1a} and \eqref{prfdiscsys1b} gives
\[
\ddt c^\gk_j(t) = -\sum\limits_{i=0}^\gk \prodl{m_{\epsilon,\eta}(u^\gk(t)) \, \basisfcn_{i,x}}{\basisfcn_{j,x}} \, \pairh{A_\delta (u^\gk(t))}{\basisfcn_i}
\]
for \(j = 0,1,\dotsc,\gk\). Hence, for \(\mathbf{a} = \left( a_i \right)_{i=0}^\gk \in \IR^{\gk+1}\) and \(j = 0,1,\dotsc,\gk\) we define
\[
f^\gk_j(\mathbf{a}) := -\sum\limits_{i=0}^\gk \prodl{m_{\epsilon,\eta}(\textstyle\sum_k{a_k \basisfcn_k}) \, \basisfcn_{i,x}}{\basisfcn_{j,x}} \, \pairh{A_\delta(\textstyle\sum_k{a_k \basisfcn_k})}{\basisfcn_i}
\]
and set
\[
\mathbf{c}^\gk(.) := \left( c^\gk_j(.) \right)_{j=0}^\gk, \quad \mathbf{f}^\gk(.) := \left( f^\gk_j(.) \right)_{j=0}^\gk, \quad \mathbf{c}^\gk_\mathbf{0} := \bigl( \prodl{\approxu}{\basisfcn_j} \bigr)_{j=0}^\gk \cma
\]
such that \eqref{prfdiscsys1} can finally be written as
\begin{subequations} \label{prfdiscsys3}
\begin{align}
\ddt \mathbf{c}^\gk(t) &= \mathbf{f}^\gk(\mathbf{c}^\gk(t)) \cma \label{prfdiscsys3a} \\
\mathbf{c}^\gk(0) &= \mathbf{c}^\gk_\mathbf{0} \fst \label{prfdiscsys3b}
\end{align}
\end{subequations}

We intend to use Peano's theorem \ref{thmpeano} to prove existence of a local solution of the system of ordinary differential equations \eqref{prfdiscsys3}. To this end, we choose an arbitrary \(r > 0\) and observe that by the definition of \(m_{\epsilon, \eta}\) and \(A_\delta\) (and for instance Lebesgue's dominated convergence theorem) \(\mathbf{f}^\gk\) is continuous on \(B_r(\mathbf{c}_\mathbf{0}^\gk)\), where
\[
B_r(\mathbf{c}_\mathbf{0}^\gk) := \lbrace \mathbf{a} \in \IR^{\gk+1}: \norm{\mathbf{a} - \mathbf{c}^N_\mathbf{0}}{\infty} \leq r \rbrace \fst
\]
Here and in the following, \(\norm{\mathbf{a}}{\infty} := \max\limits_{0\leq i \leq \gk} \abs{a_i}\) denotes the maximum norm on \(\IR^{\gk + 1}\).

Moreover, using the boundedness of \(m_{\epsilon, \eta}\) \eqref{chap1mregbound} and \(A_\delta\) \eqref{chap1Adeltabound} we may estimate
\[
\begin{split}
\abs{f^N_j(\mathbf{a})} &\leq \sum\limits_{i=0}^\gk \left[ (\frac{1}{\eta} + 1) \norml{e_{i,x}} \norml{e_{j,x}} \cdot \sum_{k=0}^\gk (1 + \delta) \norm{\mathbf{a}}{\infty} \norml{\basisfcn_{k,x}} \norml{e_{j,x}} \right] \\
&\leq C(N, \lambda_N, \eta) \, \norm{\mathbf{a}}{\infty}
\end{split}
\]
for \(0 \leq j \leq N\), such that
\begin{equation} \label{prfdiscboundf}
\norm{\mathbf{f}^\gk (\mathbf{a})}{\infty} 
\leq C(\gk, \lambda_\gk, \eta) \, \norm{\mathbf{a}}{\infty} \fst
\end{equation}
This in particular implies
\[
\sup\limits_{B_r(\mathbf{c}^N_\mathbf{0})} \norm{\mathbf{f}^\gk (\mathbf{a})}{\infty} \leq C(\gk, \lambda_\gk, \eta) \cdot (\norm{\mathbf{c}_\mathbf{0}^\gk}{\infty} + r) =: K_{0} \fst
\]


Now, by Peano's theorem \ref{thmpeano} there exists a continuously differentiable solution \(\mathbf{c}^\gk(t)\) of \eqref{prfdiscsys3} on \([0, T_0]\), where
\[
T_0 := \min \Bigl\lbrace T, \frac{r}{K_{0}} \Bigr\rbrace \fst
\]
In particular, the solution is regular enough for the formal calculations above to be justified, and we may determine \(d_j(t)\), \(j=0,1,\dotsc,N\), from \eqref{prfdiscsys1} to obtain a local solution of system \eqref{chap1discretesys}. Global solvability of \eqref{chap1discretesys} will follow from the a priori estimates that we are going to derive in the following step.


To this end, we multiply \eqref{chap1discretesysa} by \( d^\gk_j(t) \), sum over \(j=0,1,\dotsc,N\) and integrate over \((0, t)\) for arbitrary \(0 < t \leq T_0\). This yields
\begin{equation} \label{prfdiscenergypart1}
\intex{0}{t}{\prodl{\ddt u^\gk}{p^\gk} \, ds} + \intex{0}{t}{\intspace{m_{\epsilon,\eta}(u^\gk) \, \abs{p^\gk_x}^2} \, ds} = 0 \fst
\end{equation}
On the other hand, multiplication of \eqref{chap1discretesysb} by \(\ddt  c^\gk_j(t)\), summation over \(j=0,1,\dotsc,N\) and integration over \((0, t)\) leads to
\begin{equation} \label{prfdiscenergypart2}
\begin{split}
\intex{0}{t}{\prodl{p^\gk}{\ddt  u^\gk} \, ds} &= \intex{0}{t}{\intex{\Omega}{ }{ \Bigl( \Qex{u^\gk_x} + \delta u^\gk_x \Bigr) \cdot \ddt  u^\gk_x \, dx} ds} \\
&= \intex{0}{t}{\intex{\Omega}{ }{ \ddt  \Bigl( \Q{u^\gk_x} + \frac{\delta}{2} \abs{u^\gk_x}^2 \Bigr) \, dx} ds} \fst
\end{split}
\end{equation}
Hence, employing the fundamental theorem of calculus, it follows from \eqref{prfdiscenergypart1} and \eqref{prfdiscenergypart2} that
\begin{equation} \label{prfdiscenergy}
\begin{split}
\intspace{\Q{u^\gk_x(t)}} + \frac{\delta}{2} \norml{u^\gk_x(t)}^2 + \intex{0}{t}{\intspace{m_{\epsilon,\eta}(u^\gk_x) \, \abs{p^\gk_x}^2 }}{ds} &\\
\leq \intspace{\Q{u^\gk_{0,x}}} + \frac{\delta}{2} \norml{u^\gk_{0,x}}^2 &\leq C(\approxu) \cma
\end{split}
\end{equation}
where the right-hand side is bounded due to the definition of \(u^\gk_0\) in \eqref{prfdiscsys1c} and the fact \eqref{chap1boundproj}.

Now, assume that \(T_0 < T\). By \eqref{prfdiscenergy}, it holds for \(0 \leq t \leq T_0\) and \(1 \leq j \leq N\) that
\[
\begin{split}
\abs{c^\gk_j(t)}^2 &\leq \sum_{i=1}^\gk{\abs{c^\gk_i(t)}^2} \leq \frac{1}{\lambda_1} \sum_{i=1}^\gk{\lambda_i \, \abs{c^\gk_i(t)}^2 \intex{\Omega}{}{\abs{ \basisfcn_i}^2}{dx}} \\
& = \frac{1}{\lambda_1} \sum_{i=1}^\gk{\abs{c^\gk_i(t)}^2 \intex{\Omega}{}{\abs{\basisfcn_{i,x}}^2}{dx}} = \frac{1}{\lambda_1}  \intex{\Omega}{}{\sum_{i=1}^\gk {\abs{c^\gk_i(t) \basisfcn_{i,x}}^2}}{dx} \\
& = \frac{1}{\lambda_1} \intex{\Omega}{}{\abs{u^\gk_x(t)}^2}{dx} \leq C(\approxu, \delta) \cma
\end{split}
\]
where we have used the fact that \(0 < \lambda_1 < \lambda_2 < \dotsc\) and the orthogonality of \(\sequence{\basisfcn_j}{j \in \IN_0}\). Since moreover \(\ddt c^\gk_0(t) = 0\) by \eqref{prfdiscsys1a}, we have for \(0 \leq t \leq T_0\)
\[
\abs{c^\gk_0(t)} = \abs{c^\gk_0(0)} = \abs{\prodl{u_0}{e_0}} \leq C(u_0) \fst
\]
Hence, we conclude that
\begin{equation} \label{prfdiscboundc}
\norm{\mathbf{c}^\gk(t)}{\infty} \leq C(\approxu, \delta)
\end{equation}
for \(0 \leq t \leq T_0\). 
From \eqref{prfdiscboundf} and \eqref{prfdiscboundc} we now infer that
\[
\begin{split}
\sup\limits_{B_r(\mathbf{c}^N(t))} \norm{\mathbf{f(\mathbf{a})}}{\infty} &\leq C(N, \lambda_N, \eta) \cdot (\norm{\mathbf{c}^N(t)}{\infty} + r) \\ &\leq C(u_0, N, \lambda_N, \eta, \delta, r) =: K \fst
\end{split}
\]
Thus, applying Peano's theorem \ref{thmpeano} again, we may extend the local solution of \eqref{prfdiscsys3} found above as well as the energy estimate \eqref{prfdiscenergy} onto the interval \([0, T_0 + \frac{r}{K}]\). Note in particular that the constant \(K\) depends on given data and the choice of \(r > 0\) only. Hence, we may repeat this argument to obtain a solution of \eqref{prfdiscsys3} and in turn of system \eqref{chap1discretesys} on all of \([0, T]\) after finitely many steps. This proves the lemma. \qedhere

\end{proof}

We are now ready to prove Proposition \ref{proposition1}.

\begin{proof}[Proof of Proposition \ref{proposition1}]
Let \(N \in \IN\). Multiplying \eqref{chap1discretesysa} and \eqref{chap1discretesysb} by arbitrary constants and summing over \(j, k = 0, \dotsc, N\), we see that the functions \(u^N\) and \(p^N\) obtained by Lemma \ref{lemdisc} satisfy
\begin{subequations} \label{prf1discretesys}
\begin{align}
\prodl{\ddt u^\gk(t)}{\varphi} &+ \prodl{m_{\epsilon,\eta}(u^\gk(t)) \, p_x^\gk(t)}{\varphi_x} = 0 \cma \label{prf1discretesysa} \\
\prodl{p^\gk(t)}{\psi} &= \pairh{A_\delta (u^\gk(t))}{\psi} \cma \label{prf1discretesysb} \\
u^\gk(0) &= \proj_\gk \approxu \label{prf1discretesysc}
\end{align}
\end{subequations}
for arbitrary \(\varphi, \psi \in \discspace_\gk\). In addition,
\begin{equation} \label{prf1energy}
\intex{\Omega}{ }{\Q{u^\gk_x(t)}}{dx} + \frac{\delta}{2} \norml{u^\gk_x(t)}^2 + \intspacetime{ m_{\epsilon,\eta}(u^\gk_x(s)) \, \abs{p^\gk_x(s)}^2 } \leq C(\approxu)
\end{equation}
holds for \(t \in [0,T]\).

\beginstep{2}{A priori estimates.}
We are going to derive some a priori estimates from the energy estimate \eqref{prf1energy}. First of all, we use the constant function \(\varphi \equiv 1 \in \discspace_\gk\) as a test function in \eqref{prf1discretesysa} to obtain
\begin{equation} \label{prf1consmass1}
\intspace{u^\gk(t)} = \intspace{u^\gk_0}
\end{equation}
for arbitrary \(0 \leq t \leq T\). That is, the solution \(u^\gk\) conserves mass. Similarly, testing \eqref{prf1discretesysb} with \(\psi \equiv 1 \in \discspace_\gk\) yields
\begin{equation} \label{prf1consmass2}
\intspace{p^\gk(t)} = 0
\end{equation}
for \(0 \leq t \leq T\). Now, the boundedness of the second term in \eqref{prf1energy} together with \eqref{prf1consmass1} and Poincar\'e's inequality implies
\begin{equation} \label{prf1apriori1}
\norm{u^\gk}{\bochner{L^\infty}{H^1}} \leq C(\approxu, \delta) \fst
\end{equation}
Moreover, from the positivity of \(m_{\epsilon, \eta}\) \eqref{chap1mregpos}, estimate \eqref{prf1energy}, \eqref{prf1consmass2} and Poincar\'e's inequality once more we deduce
\begin{equation} \label{prf1apriori3}
\norm{p^\gk}{\bochner{L^2}{H^1}} \leq C(\approxu, \epsilon) \fst
\end{equation}
Going further, we may integrate by parts in \eqref{prf1discretesysb} to obtain
\begin{equation} \label{prf1identp}
\begin{split}
\prodl{p^\gk(t)}{\psi} &= \pairh{A_\delta (u^\gk(t))}{\psi} \\
&= - \intspace{\leftA \ddx \Qex{u^\gk_x(t)} + \delta u^\gk_{xx}(t) \rightA \cdot \psi}\\
&= - \intspace{\leftA \frac{u^\gk_{xx}(t)}{\Q{u^\gk_x(t)}^3} + \delta u^\gk_{xx}(t) \rightA \cdot \psi}
\end{split}
\end{equation}
for \(\psi \in \discspace_\gk\) and \(t \in [0,T]\). Hence, taking \(\psi = -u^\gk_{xx}(t) \in \discspace_\gk\) as a test function, integrating over \((0,T)\) and employing H\"{o}lder's and Young's inequality leads to
\[
\begin{split}
\intspacetime{\dfrac{\abs{u^\gk_{xx}}^2}{\Q{u^\gk_x}^3} + \delta \abs{u^\gk_{xx}}^2} &= - \inttime{\prodl{p^\gk}{u^\gk_{xx}}} \\
&\leq \frac{\delta}{2} \norm{u^\gk_{xx}}{L^2(\Omega_T)}^2 + C(\delta) \norm{p^\gk}{L^2(\Omega_T)} \fst
\end{split}
\]
Thus,
\begin{equation} \label{prf1apriori5}
\norm{u^\gk}{\bochner{L^2}{H^2}} \leq C(\approxu, \delta, \epsilon) \fst
\end{equation}

To prove compactness in time, we observe that \(u^\gk_t(t) = \ddt u^\gk(t) \in \discspace_\gk \) such that for arbitrary \(\zeta \in H^1(\Omega)\) and \(t \in [0,T]\)
\begin{equation}
\prodl{u^\gk_t(t)}{\zeta} = \prodl{u^\gk_t(t)}{\proj_\gk \zeta} = - \prodl{m_{\epsilon, \eta}(u^\gk(t)) \, p^\gk_x(t)}{(\proj_N \zeta)_x} \label{prf1utp}
\end{equation}
holds by equation \eqref{prf1discretesysa}. Hence, having in mind the boundedness of \(m_{\epsilon, \eta}\) \eqref{chap1mregbound}, the energy estimate \eqref{prf1energy} and the fact
\[
\norm{\proj_\gk}{\mathcal{L}\left( H^1(\Omega); H^1(\Omega) \right)} = 1 \cma
\]
we integrate relation \eqref{prf1utp} over \((0,T)\) and employ H\"{o}lder's inequality to deduce for arbitrary \(v \in \bochner{L^2}{H^1}\)
\[
\begin{split}
\absB{\inttime{\prodl{u^\gk_t}{v}}}
& \leq \intspacetime{\abs{m_{\epsilon,\eta}(u^\gk_x) \, p^\gk_x \cdot (\proj_\gk v)_x}} \\
& \leq \norm{m_{\epsilon,\eta}(u^\gk_x) \, p^\gk_x}{L^2(\Omega_T)} \norm{(\proj_\gk v)_x}{L^2(\Omega_T)} \\
& \leq (\frac{1}{\eta}+1)^\frac{1}{2} \norm{m_{\epsilon,\eta}(u^\gk_x)^\frac{1}{2} \, p^\gk_x}{L^2(\Omega_T)} \norm{(\proj_\gk v)_x}{L^2(\Omega_T)} \\
& \leq C(\approxu, \eta) \norm{v}{\bochner{L^2}{H^1}} \fst
\end{split}
\]
It follows
\[
u^\gk_t \in \bochner{L^2}{H^1}' \cong \timespace
\]
and in particular
\begin{equation} \label{prf1apriori6}
\norm{u^\gk_t}{\timespace} \leq C(\approxu, \eta) \fst
\end{equation}


\beginstep{3}{The limit \(\gk \to \infty\).}
Using the a priori estimates \eqref{prf1apriori1}, \eqref{prf1apriori3}, \eqref{prf1apriori5} and \eqref{prf1apriori6}, we infer the following convergence results for a subsequence, also denoted by \(\sequence{u^\gk}{N \in \IN}\) and \(\sequence{p^\gk}{N \in \IN}\), for \(\gk \to \infty\):
\begin{alignat}{3}
u^\gk &\weakto u &&\weaktext &&\bochner{L^2}{H^2} \label{prf1weakconvu} \\
u^\gk &\weakstarto u &&\weakstartext &&\bochner{L^\infty}{H^1} \\
u^\gk_t &\weakto u_t &&\weaktext &&\timespace \\
p^\gk &\weakto p &&\weaktext &&\bochner{L^2}{H^1} \label{prf1weakconvp}
\end{alignat}
Since \(H^2(\Omega) \cptembedding H^1(\Omega) \embedding H^1(\Omega)'\) and \(H^1(\Omega) \cptembedding L^2(\Omega) \embedding H^1(\Omega)'\), Simon's theorem \ref{thmsimon} implies for a further subsequence that
\begin{alignat}{3}
u^\gk &\strongto u &&\strongtext &&\bochner{L^2}{H^1} \label{prf1strconvu} \cma \\
u^\gk &\strongto u &&\strongtext &&\bochnerex{C}{[0,T]}{L^2(\Omega)} \label{prf1strconvuC} \fst
\end{alignat}
This proves the regularity results \eqref{prop1regu}--\eqref{prop1regp}. Since moreover \(u^\gk_x \in \bochner{L^2}{H^1_0}\) for all \(\gk \in \IN\) by construction,
\eqref{prop1boundary} follows from \eqref{prf1weakconvu}. The initial condition \eqref{prop1initial} is an immediate consequence of \eqref{prf1discretesysc}, \eqref{prf1strconvuC} and the fact that
\begin{alignat*}{3}
\proj_\gk u_0 &\strongto u_0 &&\strongtext &&L^2(\Omega)
\end{alignat*}
as \(N \to \infty\).

Going further, we deduce from the strong convergence of \(u^N\) and \(u^N_x\) in \(L^2(\Omega_T)\) \eqref{prf1strconvu} for yet a further subsequence that (cf. for instance \cite[Lemma 1.18]{alt})
\begin{alignat}{3}
u^\gk &\strongto u &&\pwaetext &&\Omega_T \label{prf1unpwae} \cma \\
u^\gk_x &\strongto u_x &&\pwaetext &&\Omega_T \label{prf1unxpwae} \fst
\end{alignat}
Using this, we will prove below that
\begin{alignat}{3}
m_{\epsilon,\eta}(u^\gk) \, p^\gk_x &\weakto m_{\epsilon,\eta}(u) \, p_x &&\weaktext &&L^2(\Omega_T) \cma \label{prf1weakconvmp} \\
- \ddx \Qex{u^N_x} &\weakto - \ddx \Qex{u_x} &&\weaktext &&L^2(\Omega_T) \label{prf1weakconvA} \fst
\end{alignat}

Now, for arbitrary \(v, w \in \bochner{L^2}{H^1}\) we may choose sequences \(\sequence{v^\gk}{N \in \IN}, \sequence{u^\gk}{N \in \IN} \subset \bochnerex{C^1}{[0,T]}{H^1(\Omega)}\) such that \(v^\gk(t), w^\gk(t) \in \discspace_\gk\) for \(\gk \in \IN\), \(t \in [0,T]\) and
\begin{alignat}{3}
v^\gk, w^\gk &\strongto v, w &&\strongtext &&\bochner{L^2}{H^1} \label{prf1strconvvw}
\end{alignat}
as \(N \to \infty\). Thus, if we integrate equations \eqref{prf1discretesysa} and \eqref{prf1discretesysb} over \((0,T)\) we obtain
\begin{subequations} \label{prf1discretesys2}
\begin{align}
\inttime{\pairh{u^\gk_t}{v^\gk}} &+ \intspacetime{m_{\epsilon,\eta}(u^\gk) \, p^\gk_x \cdot v^\gk_x} = 0 \cma \\
\inttime{\prodl{p^\gk}{w^\gk}} &= - \intspacetime{\ddx \leftA \Qex{u^\gk_x} + \delta u^\gk_x \rightA \cdot w^\gk} \label{prf1discretesys2b} \fst
\end{align}
\end{subequations}
Letting \(\gk \to \infty\) in \eqref{prf1discretesys2} and using the convergence results above, we end up with
\begin{subequations} \label{prf1sys}
\begin{align}
\inttime{\pairh{u_t}{v}} &+ \intspacetime{m_{\epsilon,\eta}(u) \, p_x \cdot v_x} = 0 \cma \\
\inttime{\prodl{p}{w}} &= - \intspacetime{\ddx \leftA \Qex{u_x} + \delta u_x \rightA \cdot w} \fst
\end{align}
\end{subequations}
In particular, by the arbitrariness of \(w \in \bochner{L^2}{H^1}\) we may indeed identify
\[
p = - \ddx \leftA \Qex{u_x} + \delta u_x \rightA
\]
in \(L^2(\Omega_T)\). This proves \eqref{prop1sys}.

It remains to prove \eqref{prf1weakconvmp} and \eqref{prf1weakconvA}. By the continuity of \( m_{\epsilon,\eta} \) and the pointwise convergence of \(u^\gk \) \eqref{prf1unpwae}, \(m_{\epsilon,\eta}(u^\gk)\) converges pointwise almost everywhere in \(\Omega_T\) to \(m_{\epsilon,\eta}(u)\) as \(N \to \infty\). Moreover, the boundedness of \(m_{\epsilon,\eta}\) \eqref{chap1mregbound} and Vitali's convergence theorem yield
\[
m_{\epsilon,\eta}(u^\gk) \, v \strongto m_{\epsilon,\eta}(u) \, v \strongtext L^2(\Omega_T)
\]
for arbitrary \(v \in L^2(\Omega_T)\). Taking the weak convergence of \(p^\gk_x\) \eqref{prf1weakconvp} into account, it follows that
\[
\lim\limits_{\gk \to \infty} \intspacetime{m_{\epsilon, \eta}(u^\gk) \, p^\gk_x \cdot v} = \intspacetime{m_{\epsilon, \eta}(u) \, p_x \cdot v}
\]
for \(v \in L^2(\Omega_T)\).
This proves \eqref{prf1weakconvmp}.

Similarly, having the relation
\[
- \ddx \Qex{u^\gk_x} = - \frac{u^\gk_{xx}}{\Q{u^\gk_x}^3}
\]
in mind, we deduce from the boundedness \(\Q{u^\gk_x}^{-3} \leq 1\), the pointwise convergence of \(u^\gk_x\) \eqref{prf1unxpwae} and once more Vitali's convergence theorem that
\[
\frac{1}{\Q{u^\gk_x}^3} \, w \strongto \frac{1}{\Q{u_x}^3} \, w \strongtext L^2(\Omega_T)
\]
for arbitrary \(w \in L^2(\Omega_T)\) as \(N \to \infty\). As above, it follows from the weak convergence of \(u^\gk_{xx}\) \eqref{prf1weakconvu} that
\[
\lim\limits_{\gk \to \infty} \intspacetime{\frac{u^\gk_{xx}}{\Q{u^\gk_x}^3} \cdot w} = \intspacetime{\frac{u_{xx}}{\Q{u_x}^3} \cdot w} \fst
\]
for \(w \in L^2(\Omega_T)\). This proves \eqref{prf1weakconvA}. For another method of identifying the nonlinear pressure, confer Remark \ref{remarkmon} below.

To establish the energy estimate \eqref{prop1energy}, note that by \eqref{prf1unpwae} and \eqref{prf1unxpwae} the functions \(u^\gk(t, .)\) and \(u^\gk_x(t, .)\) converge pointwise almost everywhere in \(\Omega\) for almost all \(t \in (0,T)\) (cf. for instance \cite[Lemma A4.9]{alt}). Thus, by Fatou's lemma and the discrete energy estimate \eqref{prf1energy} we have
\[
\intspace{\Q{u_x(t)}} \leq \liminf\limits_{\gk \to \infty} \intspace{\Q{u^\gk_x(t)}} \leq C(u_0)
\]
for almost all \(t \in (0,T)\). Furthermore, note that by the same arguments as above we may prove that
\[
m_{\epsilon,\eta}(u^\gk)^\frac{1}{2} \, p^\gk_x \weakto m_{\epsilon,\eta}(u)^\frac{1}{2} \, p_x \weaktext L^2(\Omega_T) \fst
\]
Hence, using weak lower semicontinuity of the norm in \(L^2(\Omega_T)\) we deduce from estimate \eqref{prf1energy} that
\[
\begin{split}
\intspacetime {m_{\epsilon,\eta}(u) \, \abs{p_x}^2} &= \norm{m_{\epsilon,\eta}(u)^\frac{1}{2} \, p_x}{L^2(\Omega_T)} \\
&\leq \liminf\limits_{\gk \to \infty} \, \norm{m_{\epsilon,\eta}(u^\gk)^\frac{1}{2} \, p^\gk_x}{L^2(\Omega_T)} \leq C(u_0) \fst
\end{split}
\]
The boundedness of the second term in \eqref{prop1energy} is straightforward. This proves the energy estimate \eqref{prop1energy}, and the proof of Proposition \ref{proposition1} is complete. \qedhere
\end{proof}


\begin{rem} \label{remarkmon}
As previously mentioned, the nonlinear operator \(A_\delta : H^1(\Omega) \to H^1(\Omega)'\), \(\delta > 0\), can be shown to be monotone and hemicontinuous. For instance, the monotonicity follows easily from the fact that \(A_\delta\) can be obtained as the G\^{a}teaux derivative of the nonlinear convex functional \(F_\delta:H^1(\Omega) \to \IR\) given by
\[
F_\delta(u) = \intspace{\Q{u_x} + \frac{\delta}{2} \abs{u_x}^2} \qquad \forall \, u \in H^1(\Omega)
\]
for \(\delta > 0\) (cf. \cite[Proposition 1.1, p. 158]{lions}). Hence, there is a more elegant method of identifying the nonlinear pressure \(p\) in the limit. To this end, we depart from
\[
\inttime{\prodl{p^\gk}{w^\gk}} = \inttime{\pairh{A_\delta(u^\gk)}{w^\gk}}
\]
instead of \eqref{prf1discretesys2b}. Furthermore, we may employ the energy estimate \eqref{prf1energy} to prove that \(A_\delta(u^\gk)\) is uniformly bounded in \(\timespace\). Thus, we deduce the existence of a convergent subsequence, also denoted by \(\sequence{A_\delta(u^N)}{N \in \IN}\), and may employ the monotonicity of \(A_\delta\) to prove that indeed
\begin{alignat*}{3}
A_\delta(u^\gk) &\weakto A_\delta(u) &&\weaktext &&\timespace \fst
\end{alignat*}
as \(\gk \to \infty\) (cf. for instance \cite{lions,zeidler2,ruz}). Hence, we obtain
\[
\inttime{\prodl{p}{w}} = \inttime{\pairh{A_\delta(u)}{w}}
\]
in the limit. By the regularity of \(u\), we may finally identify
\[
p = - \ddx \leftA \Qex{u_x} + \delta u_x \rightA \fst
\]
Note in particular that this method does not require the gradient \(u^\gk_x\) to converge pointwise almost everywhere in \(\Omega_T\). However, since we will always have pointwise convergence of \(u^\gk_x\) throughout this work, we will continue to use the straightforward method as in the proof of Proposition \ref{proposition1} above. \qedrem



\end{rem}

\begin{rem}
So far, we could also have used the well-known theorem of Aubin--Lions (see \cite[Theorem 5.1, p. 58]{lions}) to prove strong convergence of \(u^\gk\) in \(\bochner{L^2}{H^1}\). In this context, Proposition \ref{propevtrip} would guarantee that \(u \in \bochnerex{C}{[0,T]}{L^2}\) in the limit. However, we will necessarily need the more general theorem of Simon \ref{thmsimon} below, as it does not require any of the involved spaces to be reflexive. \qedrem
\end{rem}

\subsection{Unbounded Mobilities} \label{chapunbounded}

In the previous step, we have proved existence of a weak solution \((u^\eta, p^\eta) = (u^{\delta,\epsilon,\eta}, p^{\delta,\epsilon,\eta})\), \(0 < \delta,\epsilon,\eta \leq 1\), of the following system:
\begin{subequations} \label{chap2syseta}
\begin{align}
\inttime{\pairh{u^\eta_t}{v}} &+ \intspacetime{m_{\epsilon,\eta}(u^\eta) \; p^\eta_x \cdot v_x} = 0 &\forall& \, v \in \bochner{L^2}{H^1} \label{chap2sysetaa} \\
p^\eta &= - \ddx \leftA \Qex{u^\eta_x} + \delta u^\eta_x \rightA \label{chap2sysetab}
\end{align}
\end{subequations}
In addition, the solution satisfies the following energy estimate for almost all \(t \in (0,T)\):
\begin{equation} \label{chap2energyeta}
\begin{split}
\intex{\Omega}{ }{\Q{u^\eta_x(t)}}{dx} + \frac{\delta}{2} \norml{u^\eta_x(t)}^2 + \intspacetime{ m_{\epsilon, \eta}(u^\eta) \, \abs{p_x}^2 } \leq C(\approxu)
\end{split}
\end{equation}

The redundancy of the artificial bound on the mobility \(m_{\epsilon, \eta}\) provided by \(\eta > 0\) is obvious: Estimate \eqref{chap2energyeta} and conservation of mass imply that the solutions \(\sequence{u^\eta}{\eta > 0}\) are uniformly bounded in \(\bochner{L^\infty}{W^{1,1}}\). Since \(W^{1,1}(\Omega) \embedding L^\infty(\Omega)\) holds true in the one-dimensional setting, \(u^\eta\) and \(m_{\epsilon, \eta}(u^\eta)\) are bounded almost everywhere in \(\Omega_T\) independently of \(\eta > 0\). Relying on the energy estimate \eqref{chap2energyeta} above, we will obtain the same regularity results as in Proposition \ref{proposition1} in the limit \(\eta \limeps 0\).

For the upcoming results we define
\[
m_{\epsilon}(s) := \lim\limits_{\eta \limeps 0} m_{\epsilon, \eta}(s) = m(s) + \epsilon \qquad \forall \, s \in \IR \fst
\]

The goal of this section is the proof of the following proposition:

\begin{prop}\label{proposition2}
Assume \assumptions and let \(0 < \delta, \epsilon \leq 1\). Then a pair of functions \((u, p) = (u^{\delta,\epsilon}, p^{\delta,\epsilon})\) exists such that
\begin{align}
u &\in \bochner{L^2}{H^2} \cap \bochner{L^\infty}{H^1} \cap \bochnerex{C}{[0,T]}{L^2(\Omega)} \cma \label{prop2regu} \\
u_t &\in \timespace \cma \label{prop2regut} \\
p &\in \bochner{L^2}{H^1} \cma \label{prop2regp} \\
u_x(t) &\in H^1_0(\Omega) \text{ a.e. in } (0,T) \cma \label{prop2boundary} \\
u(0) &= \approxu \cma \label{prop2initial}
\end{align}
which solves the following system:
\begin{subequations} \label{prop2sys}
\begin{align}
\inttime{\pairh{u_t}{v}} &+ \intspacetime{m_{\epsilon}(u) \; p_x \cdot v_x} = 0 &\forall& \, v \in \bochner{L^2}{H^1} \label{prop2sysa} \\
p &= - \ddx \leftA \Qex{u_x} + \delta u_x \rightA \label{prop2sysb}
\end{align}
\end{subequations}
Furthermore, the solution satisfies the energy estimate
\begin{equation} \label{prop2energy}
\begin{split}
\intex{\Omega}{ }{\Q{u_x(t)}}{dx} + \frac{\delta}{2} \norml{u_x(t)}^2 + \intspacetime{ m_\epsilon(u_x) \, \abs{p_x}^2 } \leq C(\approxu)\\
\end{split}
\end{equation}
for almost all \(t \in (0,T)\). \qedthm
\end{prop}

\begin{proof}[Proof of Proposition \ref{proposition2}]
Let \((u^\eta, p^\eta)\) be a solution in the sense of Proposition \ref{proposition1}.

\beginstep{1}{A priori estimates.}
We may easily prove that conservation of mass holds by taking \(v \equiv 1\) as a test function in \eqref{chap2sysetaa} and using Proposition \ref{propevtrip}. That is, we have
\begin{equation} \label{prf2consmass1}
\intspace{u^\eta(t)}  = \intspace{u_0}
\end{equation}
for all \(t \in [0,T]\). Moreover, we infer from equation \eqref{chap2sysetab} and the boundary regularity of \(u^\eta\) \eqref{prop1boundary} that
\begin{equation} \label{prf2consmass2}
\intspace{p^\eta(t)} = 0
\end{equation}
for almost all \(t \in (0,T)\).

From the energy estimate \eqref{chap2energyeta}, conservation of mass \eqref{prf2consmass1} and Poincar\'e's inequality we deduce that
\begin{align}
\norm{u^\eta}{\bochner{L^\infty}{W^{1,1}}} &\leq C(\approxu) \cma \label{prf2apriori0} \\ 
\norm{u^\eta}{\bochner{L^\infty}{H^1}} &\leq C(\approxu, \delta) \label{prf2apriori1} \fst
\end{align}
Moreover, using \eqref{chap2energyeta} and \eqref{prf2consmass2} together with Poincar\'e's inequality once again, we obtain
\begin{gather}
\norm{p^\eta}{\bochner{L^2}{H^1}} \leq C(\approxu, \epsilon) \fst \label{prf2apriori3}
\end{gather}
Multiplying \eqref{chap2sysetab} by \(-u^\eta_{xx} \in L^2(\Omega_T)\), calculation as in the proof of Proposition \ref{proposition1} yields
\begin{equation} \label{prf2bounduxx}
\intspacetime{\dfrac{\abs{u^\eta_{xx}}^2}{\Q{u^N_x}^3} + \delta \abs{u^\eta_{xx}}^2 } \leq \frac{\delta}{2} \norm{u^\eta_{xx}}{L^2(\Omega_T)}^2 + C(\delta) \norm{p^\eta}{L^2(\Omega_T)} \fst 
\end{equation}
Thus,
\begin{gather}
\norm{u^\eta}{\bochner{L^2}{H^2}} \leq C(\approxu, \delta, \epsilon) \fst \label{prf2apriori4}
\end{gather}
Finally, note that \(W^{1,1}(\Omega) \hookrightarrow L^\infty(\Omega)\) and moreover \(\bochner{L^\infty}{L^\infty} \subseteq L^\infty(\Omega_T)\). Hence, by \eqref{prf2apriori0},
\begin{align}
\norm{u^\eta}{L^\infty(\Omega_T)} &\leq C(\approxu) \fst \label{prf2boundu}
\end{align}
Using \eqref{prf2boundu}, the energy estimate \eqref{chap2energyeta} and H\"{o}lder's inequality, we may proceed similarly as in the previous proof to obtain
\[
\begin{split}
\absB{\inttime{{\pairh{u^\eta_t}{v}}}} &\leq \intspacetime{\abs{m_{\epsilon, \eta}(u^\eta) \, p^\eta_x \cdot v_x}} \\
&\leq \norm{m_{\epsilon, \eta}(u^\eta) \, p^\eta_x}{L^2(\Omega_T)} \norm{v_x}{L^2(\Omega_T)} \\
&\leq \norm{m_{\epsilon, \eta}(u^\eta)}{L^\infty(\Omega_T)}^{\frac{1}{2}} \norm{m_{\epsilon, \eta}(u^\eta)^\frac{1}{2} \, p^\eta_x}{L^2(\Omega_T)} \norm{v_x}{L^2(\Omega_T)} \\
&\leq C(u_0, m) \norm{v}{\bochner{L^2}{H^1}}
\end{split}
\]
such that
\begin{gather}
\norm{u^\eta_t}{\timespace} \leq C(\approxu) \fst \label{prf2apriori6}
\end{gather}


\beginstep{2}{The limit \(\eta \limeps 0\).}
Using the a priori estimates \eqref{prf2apriori1}, \eqref{prf2apriori3}, \eqref{prf2apriori4}, \eqref{prf2apriori6} and applying Simon's theorem \ref{thmsimon} in the fashion above, we infer the following convergence results for a subsequence, also denoted by \(\sequence{u^\eta}{\eta > 0}\) and \(\sequence{p^\eta}{\eta > 0}\), for \(\eta \limeps 0\):
\begin{alignat}{3}
u^\eta &\weakto u &&\weaktext &&\bochner{L^2}{H^2} \label{prf2weakconvu} \\
u^\eta &\weakstarto u &&\weakstartext &&\bochner{L^\infty}{H^1} \\
u^\eta &\strongto u &&\strongtext &&\bochner{L^2}{H^1} \label{prf2strconv} \\
u^\eta &\strongto u &&\strongtext &&\bochnerex{C}{[0,T]}{L^2(\Omega)} \\
u^\eta_t &\weakto u_t &&\weaktext &&\timespace \\
p^\eta &\weakto p &&\weaktext &&\bochner{L^2}{H^1} \label{prf2weakconvp}
\end{alignat}
This proves the regularity results \eqref{prop2regu}--\eqref{prop2regp}. \eqref{prop2boundary} and \eqref{prop2initial} follow from the very same arguments as in the proof of Proposition \ref{proposition1}.

As before, the strong convergence of \(u^\eta\) and \(u^\eta_x\) in \(L^2(\Omega_T)\) \eqref{prf2strconv} implies for a further subsequence that
\begin{alignat}{3}
u^\eta &\strongto u &&\pwaetext &&\Omega_T \cma \label{prf2convupwae} \\
u^\eta_x &\strongto u_x &&\pwaetext &&\Omega_T \fst \label{prf2convuxpwae}
\end{alignat}
Hence, we may argue as in the proof of Proposition \ref{proposition1} and deduce from \eqref{prf2weakconvp}, \eqref{prf2convupwae}, the boundedness of \(m_{\epsilon,\eta}(u^\eta)\) \eqref{prf2boundu} and Vitali's convergence theorem that
\begin{alignat*}{3}
m_{\epsilon, \eta}(u^\eta) \, p^\eta_x &\weakto m_\epsilon(u) \, p_x &&\weaktext &&L^2(\Omega_T) \fst
\end{alignat*}

Similarly, by the arguments in the proof of Proposition \ref{proposition1}, \eqref{prf2weakconvu} and \eqref{prf2convuxpwae} yield
\begin{alignat*}{3}
- \ddx \Qex{u^\eta_x} &\weakto - \ddx \Qex{u_x} &&\weaktext &&L^2(\Omega_T) \fst
\end{alignat*}

Letting \(\eta \limeps 0\) in \eqref{chap2syseta} and using the convergence results above, equations \eqref{prop2sys} hold.

To prove the energy estimate \eqref{prop2energy}, we argue as in the proof of Proposition \ref{proposition1}, using Fatou's lemma and weak lower semicontinuity of the norm. \qedhere
\end{proof}

\begin{rem} \label{remarkreg}
In \eqref{prf2bounduxx}, we employed a straightforward calculation to infer an uniform \(\bochner{L^2}{H^2}\)-bound on \(\sequence{u^\eta}{\eta > 0}\), eventually relying on the uniform \(\bochner{L^2}{H^1}\)-bound on \(\sequence{p^\eta}{\eta > 0}\) provided by the energy estimate as long as \(\epsilon > 0\). As an alternative argument, which would work in arbitrary space dimensions also, one observes that \(u^\eta\) solves
\[
A_\delta(u^\eta(t)) = p^\eta(t) \text{ in } H^1(\Omega)'
\]
for almost all \(t \in (0,T)\) and may henceforth apply regularity theory for second-order nonlinear elliptic equations (cf. for instance \cite{giusti}) to deduce the estimate
\[
\norm{u^\eta(t)}{H^2(\Omega)} \leq C(\delta) \norml{p^\eta(t)}
\]
for almost all \(t \in (0,T)\). This implies \eqref{prf2apriori4}. Of course, either approach is heavily depending on the condition \(\delta > 0\). However, this observation will become redundant in the next step when we let \(\epsilon \to 0\) and an additional estimate has to be derived in order to control \(p\) and spatial derivatives of \(u\). Finally, let us mention that in the case of the standard thin-film equation, where the pressure \(p^\eta = - \triangle u^\eta\) is linear, the very same regularity argument would yield an estimate in the space \(\bochner{L^2}{H^3}\), whereas an uniform \(\bochner{L^2}{H^2}\)-bound is the best one can expect in our setting due to the nonlinearity.
\qedrem
\end{rem}

\clearsection

\section{Degenerate Mobilities} \label{secdegenerate}

By Proposition \ref{proposition2}, there exists a solution \((u^\epsilon, p^\epsilon) = (u^{\delta,\epsilon}, p^{\delta,\epsilon})\), \(0 < \delta,\epsilon \leq 1\), of the following system:
\begin{subequations} \label{chap3sysepsilon}
\begin{align}
\inttime{\pairh{u^\epsilon_t}{v}} &+ \intspacetime{m_\epsilon(u^\epsilon) \, p^\epsilon_x \cdot v_x} = 0 &\forall& \, v \in \bochner{L^2}{H^1} \label{chap3sysepsilona} \\ 
p^\epsilon &= - \ddx \leftA \Qex{u^\epsilon_x} + \delta u^\epsilon_x \rightA \label{chap3sysepsilonb}
\end{align}
\end{subequations}
Furthermore, the solution satisfies the following energy estimate for almost all \(t \in (0,T)\):
\begin{equation} \label{chap3energyepsilon}
\begin{split}
\intex{\Omega}{ }{\Q{u^\epsilon_x(t)}}{dx} + \frac{\delta}{2} \norml{u^\epsilon_x(t)}^2 + \intspacetime{ m_{\epsilon}(u^\epsilon) \, \abs{p^\epsilon_x}^2 } \leq C(\approxu)
\end{split}
\end{equation}

We intend to let \(\epsilon \limeps 0 \) in the next step. Note that we would lose control of \(\sequence{p^\epsilon_x}{\epsilon > 0}\) and \(\sequence{u^\epsilon_{xx}}{\epsilon > 0}\), if we would rely on the energy estimate \eqref{chap3energyepsilon} only. As a remedy, we will prove that the solutions \(\sequence{u^\epsilon}{\epsilon > 0}\) obtained by Proposition \ref{proposition2} are H\"{o}lder-continuous on \(\overline{\Omega}_T\). Hence, \(u^\epsilon\) will converge uniformly on \(\overline{\Omega}_T\) to a continuous function \(u\) as \(\epsilon \limeps 0\). Using this result, we may show that indeed \(p_x \in L^2_{loc}(\posset{\abs{u}}_T)\) in the limit. Moreover, we will derive a new a priori estimate, the so-called entropy estimate. It will help controlling second-order derivatives of \(\sequence{u^\epsilon}{\epsilon > 0}\) and henceforth identifying the pressure \(p\) in the limit. In addition, it is the key to prove nonnegativity results for the solution \(u\).

\subsection{Continuous Solutions}

By the energy estimate \eqref{chap3energyepsilon}, we have for any solution \(u^\epsilon\) in the sense of Proposition \ref{proposition2} that
\[
\norml{u^\epsilon_x(t)} \leq C(\approxu, \delta)
\]
for almost all \(t \in (0,T)\). In addition, we may take \(v \equiv 1\) as a test function in \eqref{chap3sysepsilona} and use Proposition \ref{propevtrip} to show that conservation of mass holds.
Hence, by Poincar\'e's inequality we obtain
\begin{equation} \label{chap3boundux2}
\norm{u^\epsilon(t)}{H^1(\Omega)} \leq C(\approxu, \delta)
\end{equation}
for almost all \(t \in (0,T)\). Since moreover \(u^\epsilon \in \bochnerex{C}{[0,T]}{L^2(\Omega)}\) by \eqref{prop2regu}, we may easily deduce that \(u^\epsilon(t) \in H^1(\Omega)\) is uniquely determined for every \(t \in [0,T]\) and that \eqref{chap3boundux2} is in fact valid for all \(t \in [0,T]\). Finally, it follows from the continuous embedding \(C^{0,\frac{1}{2}}(\overline{\Omega}) \hookrightarrow H^1(\Omega)\) that
\begin{equation} \label{chap3boundux3}
\sup\limits_{t \in [0,T]} \norm{u^\epsilon(t)}{C^{0,\frac{1}{2}}(\overline{\Omega})} \leq C(\approxu, \delta) =: K \fst
\end{equation}
Note, that the constant is independent of \(\epsilon > 0\), but may depend on \(\delta > 0\).

The following lemma was introduced by Bernis and Friedman (see \cite[Lemma 2.1]{berfri}). We may basically adopt their proof for our case, since it does not involve the structure of the pressure \(p^\epsilon\) at all. However, we will differ slightly in one detail, making use of Proposition \ref{propevtrip} once more.

\begin{lem} \label{lemcont}
Let \(0 < \delta, \epsilon \leq 1\) and \(u^\epsilon = u^{\delta,\epsilon}\) be a solution in the sense of Proposition \ref{proposition2}. Then, there is a constant \(M(\delta)\) independent of \(\epsilon > 0\) such that
\begin{equation}
\abs{u^\epsilon(t_2, x) - u^\epsilon(t_1, x)} \leq M(\delta) \, \abs{t_2-t_1}^{\frac{1}{8}}
\end{equation}
for all \(x \in \overline{\Omega}\) and \(t_1, t_2 \in [0, T]\). \qedthm
\end{lem}

\begin{proof}
We argue by contradiction and suppose that for all \(M > 0\) there exist \(x_0 \in \overline{\Omega}\) and \(t_1, t_2 \in [0,T]\) such that
\begin{equation}
\abs{u^\epsilon(t_2, x_0) - u^\epsilon(t_1, x_0)} > M \, \abs{t_2-t_1}^{\frac{1}{8}} \fst \label{prfcont1}
\end{equation}
Without loss of generality we may assume \(u^\epsilon(t_2, x_0) \geq u^\epsilon(t_1, x_0)\) such that \eqref{prfcont1} is equivalent to
\begin{equation}
u^\epsilon(t_2, x_0) - u^\epsilon(t_1, x_0) > M \, \abs{t_2-t_1}^\beta \cma \label{prfcont2}
\end{equation}
where \(x_0 \in \overline{\Omega}\), \(t_1, t_2 \in [0,T]\) and \(\beta := \frac{1}{8}\).
We now choose \(\varphi(t, x) = \theta(t) \, \xi(x)\) as a test function in \eqref{chap3sysepsilona}, where the functions \(\xi\) and \(\theta\) will be defined in the following.

We set
\[
\xi(x) := \xi_0\left( \dfrac{x-x_0}{\frac{M^2}{16K^2}\abs{t_2-t_1}^{2\beta}} \right) \cma
\]
where \(M\) is the constant from \eqref{prfcont2}, \(K\) the H\"{o}lder-constant from \eqref{chap3boundux3} and \(\xi_0\) is a smooth function satisfying \(\xi_0(x) = \xi_0(-x)\), \(\xi_0(x) = 1\) for \(0 \leq x < \frac{1}{2}\), \(\xi_0(x) = 0\) for \(x \geq 1\) and \(\xi'_0(x) \leq 0\) for \(x \geq 0\). In particular, we have
\[
\xi(x) =
\begin{cases}
0 & \text{ if \(\abs{x-x_0} \geq \frac{M^2}{16K^2}\abs{t_2-t_1}^{2\beta}\)} \cma \\
1 & \text{ if \(\abs{x-x_0} \leq \frac{M^2}{2 \cdot 16K^2}\abs{t_2-t_1}^{2\beta}\)} \fst
\end{cases} 
\]
Furthermore, we may assume that \(t_1 < t_2\) (otherwise replace \(\theta\) by \(-\theta\) below) and set
\[
\theta(t) := \chi_{(t_1, t_2)}(t) \fst
\]
Obviously it holds that \(\varphi(t, x) = \theta(t) \, \xi(x) \in \bochner{L^2}{H^1}\). Using \(\varphi\) as a test function in \eqref{chap3sysepsilona} leads to
\begin{equation}
\inttime{\pairh{u^\epsilon_t}{\xi(.)}\theta(t)} = - \intspacetime{m_\epsilon(u^\epsilon) \, p^\epsilon_x \cdot \xi'(x) \, \theta(t)} \fst \label{prfcont4}
\end{equation}
For the left-hand side, we observe that \(1 \cdot \xi \in \bochnerex{C^1}{[0,T]}{H^1(\Omega)}\), and by the rule of partial integration as in Proposition \ref{propevtrip} (iv) we obtain
\begin{equation}
\begin{split}
\inttime{\pairh{u^\epsilon_t}{\xi(.)}\theta(t)} &= \int\limits_{t_1}^{t_2}{\pairh{u^\epsilon_t}{\xi(.)}} \, dt \\
&= \intspace{\bigl( u^\epsilon(t_2, x) -  u^\epsilon(t_1, x) \bigr) \cdot \xi(x)} \fst \label{prfcont5}
\end{split}
\end{equation}
We estimate the right-hand side of \eqref{prfcont5} from below: Since \(\xi(x) = 0\) for \(\abs{x-x_0} \geq \frac{M^2}{16K^2}\abs{t_2-t_1}^{2\beta}\), we only need to consider \(x \in \overline{\Omega}\) with \(\abs{x-x_0} \leq \frac{M^2}{16K^2}\abs{t_2-t_1}^{2\beta}\). For such \(x\) it holds by \eqref{chap3boundux3} and \eqref{prfcont2} that
\[
\begin{split}
u^\epsilon(t_2, x) - u^\epsilon(t_1, x) & = \bigl( u^\epsilon(t_2, x) - u^\epsilon(t_2, x_0) \bigr) \\
& \quad + \bigl( u^\epsilon(t_2, x_0) - u^\epsilon(t_1, x_0) \bigr) \\
& \quad + \bigl( u^\epsilon(t_1, x_0) - u^\epsilon(t_1, x) \bigr) \\
& \geq -2K \abs{x-x_0}^{\frac{1}{2}} + M \abs{t_2-t_1}^\beta \\
& \geq \dfrac{M}{2} \abs{t_2-t_1}^\beta \fst
\end{split}
\]
Moreover, we set \( E := \left\{x \in \Omega : \xi(x) = 1 \right\}\) and observe that \(\mu(E) \geq \frac{M^2}{2 \cdot 16K^2}\abs{t_2-t_1}^{2\beta}\) independently of \(x_0 \in \overline{\Omega}\). Thus,
\begin{equation}
\begin{split}
\intspace{\bigl( u^\epsilon(t_2, x) -  u^\epsilon(t_1, x) \bigr) \cdot \xi(x) } & \geq \int\limits_E {\bigl( u^\epsilon(t_2, x) -  u^\epsilon(t_1, x) \bigr) \cdot \xi(x) \, dx} \\
& \geq \int\limits_E {\frac{M}{2} \abs{t_2-t_1}^\beta \cdot \xi(x) \, dx} \\
& \geq \frac{M}{2} \abs{t_2-t_1}^\beta \cdot \frac{M^2}{2 \cdot 16K^2}\abs{t_2-t_1}^{2\beta} \\
& = \frac{M^3}{4 \cdot 16K^2}\abs{t_2-t_1}^{3\beta} \fst
\end{split} \label{prfcont7}
\end{equation}
On the other hand, using H\"{o}lder's inequality, Fubini's theorem and the energy estimate \eqref{chap3energyepsilon}, we estimate for the right-hand side of \eqref{prfcont4}
{\allowdisplaybreaks[3]
\begin{align*}
& \Bigl\vert \intspacetime{m_\epsilon(u^\epsilon) \, p^\epsilon_x \cdot \xi_x(x) \, \theta(t)} \Bigr\vert \\
& \leq \Bigl( \intspacetime{\abs{ m_\epsilon(u^\epsilon) \, p^\epsilon_x }^2} \Bigr)^\frac{1}{2} \Bigl( \intspacetime{\abs{ \xi'(x)
\theta(t) }^2} \Bigr)^\frac{1}{2} \\
& = \Bigl( \intspacetime{\abs{ m_\epsilon(u^\epsilon) \, p^\epsilon_x }^2} \Bigr)^\frac{1}{2} \Bigl( \intspace{\abs{ \xi'(x) }^2} \Bigr)^\frac{1}{2} \Bigl( \inttime{\abs{ \theta(t) }^2} \Bigr)^\frac{1}{2} \\
& \leq C(u_0) \cdot \sup_{x \in \Omega}{\abs{ \xi'(x) }} \cdot \mu \left( \supp (\xi) \right)^\frac{1}{2} \cdot \abs{t_2-t_1}^\frac{1}{2} \\
& \leq C(u_0) \cdot \dfrac{C(\xi_0)}{\frac{M^2}{16K^2}\abs{t_2-t_1}^{2\beta}} \cdot \dfrac{\sqrt{2}M}{4K}\abs{t_2-t_1}^\beta \cdot \abs{t_2-t_1}^\frac{1}{2} \\
& \leq C_1(u_0, \xi_0, K) \cdot \dfrac{1}{M} \abs{t_2-t_1}^{\frac{1}{2}-\beta} \fst
\end{align*}
}
Thus, together with \eqref{prfcont4} and \eqref{prfcont7} we obtain
\[
M^3 \abs{t_2-t_1}^{3\beta} \leq C_2(u_0, \xi_0, K) \dfrac{1}{M} \abs{t_2-t_1}^{\frac{1}{2}-\beta} \fst
\]
But since \(\beta = \frac{1}{8}\), this is equivalent to
\[
M^4 \leq C_2(u_0, \xi_0, K) \abs{t_2-t_1}^{\frac{1}{2}-4\beta} = C_2(u_0, \xi_0, K) \fst
\]
Hence, \(M\) is bounded by a constant independent of \(x_0, t_1\) and \(t_2\). This contradicts the assumption and the lemma is proved.
\end{proof}

\begin{cor} \label{corcont}
Let \(0 < \delta, \epsilon \leq 1\) and \(u^\epsilon = u^{\delta, \epsilon}\) be a solution in the sense of Proposition \ref{proposition2}. Then \(u^\epsilon\) is continuous on \(\overline{\Omega}_T\) and satisfies
\begin{align}
\abs{u^\epsilon(t_0,x_0)} &\leq C_0(u_0) \cma \label{prfcont2holder0} \\
\abs{u^\epsilon(t_2, x_2) - u^\epsilon(t_1, x_1)} &\leq C_1(u_0, \delta) \, \bigl( \abs{t_2 - t_1}^\frac{1}{8} +  \abs{x_2 - x_1}^\frac{1}{2} \bigr) \label{prfcont2holder1}
\end{align}
for all \((t_i, x_i) \in \overline{\Omega}_T\), \(i = 0, 1, 2\). \qedthm
\end{cor}

\begin{proof}
The continuity and the H\"{o}lder-condition \eqref{prfcont2holder1} follow from \eqref{chap3boundux3} and Lemma \ref{lemcont}. As in \eqref{prf2boundu}, we deduce from the energy estimate \eqref{chap3energyepsilon} and conservation of mass that
\[
\norm{u^\epsilon}{L^\infty(\Omega_T)} \leq C(u_0)
\]
holds independently of \(\delta > 0\) and \(\epsilon > 0\), and \eqref{prfcont2holder0} follows.
\end{proof}

\subsection{The Entropy Estimate} \label{chapentropy}

For \(0 < \epsilon \leq 1\) we consider the functions
\begin{align}
g_\epsilon(s) & := -\intex{s}{a}{\dfrac{1}{m_{\epsilon}(r)}}{dr} \cma \label{chap3defg} \\
G_\epsilon(s) & := -\intex{s}{a}{g_\epsilon(r)}{dr} \cma \label{chap3defG}
\end{align}
where we choose
\begin{equation} \label{chap3defa}
a > \max\limits_{(t,x) \in \overline{\Omega}_T} \abs{u^\epsilon(t,x)} \fst
\end{equation}
Note in particular that we may choose \(a\) independently of \(\epsilon > 0\) due to Corollary \ref{corcont}. Obviously it holds that
\begin{equation}
G_\epsilon(s) \geq 0, \qquad G_\epsilon'(s) = g_\epsilon(s) \leq 0 \qquad \forall \, s \leq a \label{chap3gmon}
\end{equation}
and we may estimate
\[
g_\epsilon(s) \leq \dfrac{1}{\epsilon} \, \abs{s - a}, \quad G_\epsilon(s) \leq \dfrac{1}{\epsilon} \, \abs{s - a}^2 \fst
\]
Furthermore, we have for \(0 < \epsilon_1 < \epsilon_2\)
\begin{equation} \label{chap3Gmon}
G_{\epsilon_2} (t) \leq G_{\epsilon_1} (t) \leq G(t) := \lim_{\epsilon \to 0} G_{\epsilon}(t) \fst
\end{equation}

Since \(g_\epsilon \in C^1(\IR)\) and \(g'_\epsilon = \dfrac{1}{m_\epsilon(.)} \in L^\infty(\IR)\) for \(\epsilon > 0\), we may apply the chain rule for weak derivatives (cf. for instance \cite[Lemma 7.5]{giltru}) to obtain
\[
\bigl(g(u^\epsilon)\bigr)_x = g'_\epsilon(u^\epsilon) \, u^\epsilon_x = \dfrac{1}{m_\epsilon(u^\epsilon)} u^\epsilon_x
\]
and deduce that \(g_\epsilon(u^\epsilon) \in \bochner{L^2}{H^1}\) holds by the regularity of \(u^\epsilon\). Using \(g_\epsilon(u^\epsilon) \cdot \chi_{(t_1, t_2)}\) for \(0\leq t_1, t_2 \leq T\) as a test function in \eqref{chap3sysepsilona} leads to
\begin{equation} \label{chap3entropyepsilon1}
\begin{split}
0 &= \inttimeex{t_1}{t_2}{\pairh{u^\epsilon_t}{g_\epsilon(u^\epsilon)}} + \intspacetimeex{(t_1, t_2) \times \Omega}{m_\epsilon(u^\epsilon) \, p^\epsilon_x \cdot \left(g_\epsilon(u^\epsilon)\right)_x} \\ 
&= \inttimeex{t_1}{t_2}{\pairh{u^\epsilon_t}{g_\epsilon(u^\epsilon)}} + \intspacetimeex{(t_1, t_2) \times \Omega}{ p^\epsilon_x \cdot u^\epsilon_x} \fst
\end{split}
\end{equation}
On the other hand, we may multiply \eqref{chap3sysepsilonb} by \(-u^\epsilon_{xx} \cdot \chi_{(t_1, t_2)} \in L^2(\Omega_T)\) and integrate by parts to obtain
\begin{equation} \label{chap3entropyepsilon2}
\begin{split}
\intspacetimeex{(t_1, t_2) \times \Omega}{ p^\epsilon_x \cdot u^\epsilon_x} &= -\inttimeex{t_1}{t_2}{\prodl{p^\epsilon}{u^\epsilon_{xx}}} \\
&= \intspacetimeex{(t_1, t_2) \times \Omega}{\dfrac{\abs{u^\epsilon_{xx}}^2}{\Q{u^\epsilon_x}^3} + \delta \abs{u^\epsilon_{xx}}^2}
\end{split}
\end{equation}
Note that the boundary terms vanish due to the regularity of \(u^\epsilon\) provided by \eqref{prop2boundary}.
Putting together \eqref{chap3entropyepsilon1} and \eqref{chap3entropyepsilon2} we arrive at
\begin{equation} \label{chap3entropyepsilon3}
\inttime{\pairh{u^\epsilon_t}{g_\epsilon(u^\epsilon)}} + \intspacetimeex{(t_1, t_2) \times \Omega}{\dfrac{\abs{u^\epsilon_{xx}}^2}{\sqrt{1+ \abs{u^\epsilon_x}^2}^3} + \delta \abs{u^\epsilon_{xx}}^2} = 0
\end{equation}
for \(0\leq t_1, t_2 \leq T\). We prove the following lemma:

\begin{lem} \label{lemchain}
Let \(G \in C^2(\IR)\) and denote \(g := G'\). Moreover, suppose that \(\abs{g'} \leq c_0\). Then, for \(u \in \evspace\) and arbitrary \(0\leq t_1, t_2 \leq T\) it holds that
\[
\intspace{G(u(t_2))} - \intspace{G(u(t_1))} = \intex{t_1}{t_2}{\pairh{u_t}{g(u)}}{dt} \fst
\]
\qedthm
\end{lem}

\begin{proof}
By Proposition \ref{propevtrip} (iii) we may choose a sequence \(\sequence{u^k}{k \in \IN} \subset \bochnerex{C^1}{[0,T]}{H^1(\Omega)}\) such that
\begin{alignat}{3}
u^k &\strongto u &&\strongtext &&\bochner{L^2}{H^1} \cma \label{prfchainstrconvu} \\
u^k_t &\strongto u_t &&\strongtext &&\bochnerex{L^2}{\openI}{H^1(\Omega)'} \label{prfchainstrconvut}
\end{alignat}
as \(k \to \infty\). In view of \eqref{prfchainstrconvu} we may in particular assume that \(u^k\) and \(u^k_x\) converge pointwise almost everywhere in \(\Omega_T\) as \(k \to \infty\). Furthermore, from the continuity of the embedding \(\evspace \embedding \bochnerex{C}{[0,T]}{L^2(\Omega)}\) (cf. Proposition \ref{propevtrip} (ii)) we deduce that 
\begin{alignat}{3}
u^k(t) &\strongto u(t) &&\strongtext &&L^2(\Omega) \label{prfchainstrconvuc}
\end{alignat}
for all \(t \in [0,T]\). Moreover, note that by the boundedness of \(g'\) the following growth conditions on \(g\) and \(G\) hold:
\begin{alignat*}{3}
\abs{g(s)} &\leq c_1 \abs{s} + c_2 \qquad &\forall& \, s \in \IR\\
\abs{G(s)} &\leq c_3 \abs{s}^2 + c_4 \qquad &\forall& \, s \in \IR
\end{alignat*}
Hence, using Vitali's convergence theorem and \eqref{prfchainstrconvu} we infer
\begin{alignat}{3}
g(u^k) &\strongto g(u) &&\strongtext &&\bochner{L^2}{H^1} \label{prfchainstrconvg}
\end{alignat}
as \(k \to \infty\).

We define the functional
\[
\mathbf{G} = \intspace{G(.)} : L^2(\Omega) \to \IR \fst
\]
By the differentiability of \(G\) and the boundedness of \(g'\) we get
\[
\intspace{G(v+h)} - \intspace{G(v)} = \intspace{g(v) \cdot h} + o\bigl(\norml{h}\bigr) \cma
\]
for arbitrary \(v, h \in L^2(\Omega)\). Hence, \(\mathbf{G}\) is differentiable in the sense of Fr\'echet, and since \(u^k \in \bochnerex{C^1}{[0,T]}{H^1(\Omega)}\) we may apply the chain rule for Fr\'echet-derivatives (see \cite[Proposition 4.10]{zeidler1}) to \(\mathbf{G} \circ u^k : [0,T] \to \IR \) and obtain
\[
\frac{d}{dt} \intspace{G(u^k(t))} = \intspace{g(u^k(t)) \cdot u^k_t(t)}
\]
for arbitrary \(t \in [0,T]\). The fundamental theorem of calculus yields
\begin{equation} \label{prfchainapprox}
\intspace{G(u^k(t_2))} - \intspace{G(u^k(t_1))} = \intex{t_1}{t_2}{\prodl{g(u^k)}{u^k_t}}{dt}
\end{equation}
for \(t_1, t_2 \in [0,T]\).

Letting \(k \to \infty\) in \eqref{prfchainapprox}, we observe that by \eqref{prfchainstrconvut} and \eqref{prfchainstrconvg}
\[
\intex{t_1}{t_2}{\prodl{g(u^k)}{u^k_t}}{dt} = \intex{t_1}{t_2}{\pairh{u^k_t}{g(u^k)}}{dt} \xrightarrow{k \to \infty} \intex{t_1}{t_2}{\pairh{u_t}{g(u)}}{dt}
\]
and by \eqref{prfchainstrconvuc} together with Vitali's convergence theorem
\[
\intspace{G(u^k(t))} \xrightarrow{k \to \infty} \intspace{G(u(t))}
\]
for any \(t \in [0, T]\). This proves the lemma.
\end{proof}

Now, note that by \(\eqref{chap3Gmon}\)
\[
\intspace{G_\epsilon(u_0)} \leq \intspace{G(u_0)} \leq C(u_0)
\]
independently of \(\epsilon > 0\) due to assumption (H2). Hence, combining \eqref{chap3entropyepsilon3} and Lemma \ref{lemchain} applied to the functions \(G_{\epsilon}\) yields the estimate
\begin{equation} \label{chap3entropyepsilon}
\sup\limits_{t \in [0,T]} \intspace{G_\epsilon(u^\epsilon(t))} + \intspacetime{\dfrac{\abs{u^\epsilon_{xx}}^2}{\sqrt{1+ \abs{u^\epsilon_x}^2}^3} + \delta \abs{u^\epsilon_{xx}}^2} \leq C(u_0)
\end{equation}
for \(0 < \epsilon \leq 1\), which will in the following be referred to as ``entropy estimate''.

\begin{rem}
In \cite{gruen1}, the author uses a very similar argument to derive the entropy estimate and, in addition, to reestablish the energy estimate in the setting of the standard thin-film equation in higher space dimensions. Regarding the latter, Lemma \ref{lemchain} seems to be well suited for this purpose in our case also, since the function
\[
s \mapsto F_\delta(s) = \Q{s} + \frac{\delta}{2} \abs{s}^2
\]
is convex and meets the conditions of Lemma \ref{lemchain} for \(\delta > 0\). However, note that the regularity of \(u^\epsilon_x\) is too weak to apply Lemma \ref{lemchain} or a similar argument directly. Moreover, it turns out to be rather difficult to mimic the above approximation argument because of the coupling and the nonlinearity of the pressure \(p^\epsilon\). This is the reason why we spend careful effort on preserving the energy estimate in each limit. \qedrem
\end{rem}




\subsection{Existence}

Using the continuity of \(u^\epsilon\) and the entropy estimate \eqref{chap3entropyepsilon}, we will now prove the following proposition:

\begin{prop}\label{proposition3}
Assume \assumptions and let \(0 < \delta \leq 1\). Then a pair of functions \((u, p) = (u^{\delta}, p^{\delta})\) exists such that
\begin{align}
u &\in \bochner{L^2}{H^2} \cap \bochner{L^\infty}{H^1} \cap \bochnerex{C}{[0,T]}{L^2(\Omega)} \cma \label{prop3regu} \\
u_t &\in \timespace \cma \label{prop3regut} \\
p &\in L^2(\Omega_T), \quad p_x \in L^2_{loc}(\possettimeabs{u}) \label{prop3regp} \\
u_x(t) &\in H^1_0(\Omega) \text{ a.e. in } (0,T) \cma \label{prop3boundary} \\
u(0) &= \approxu \label{prop3initial} \cma
\end{align}
which solves the following system:
\begin{subequations} \label{prop3sys}
\begin{align}
\inttime{\pairh{u_t}{v}} &+ \intspacetimeex{\possettimeabs{u}}{m(u) \, p_x \cdot v_x} = 0 &\forall& \, v \in \bochner{L^2}{H^1} \label{prop3sysa} \\
p &= - \ddx \leftA \Qex{u_x} + \delta u_x \rightA \label{prop3sysb}
\end{align}
\end{subequations}
Furthermore, the solution satisfies the energy estimate
\begin{equation} \label{prop3energy}
\begin{split}
\intspace{\Q{u_x(t)}} + \frac{\delta}{2} \norml{u_x(t)}^2 + \intspacetimeex{\possettimeabs{u}}{m(u)^r \, \abs{p_x}^2} \leq C(\approxu, r)
\end{split}
\end{equation}
for almost all \(t \in (0,T)\) and arbitrary \(r > 1\), and the estimate
\begin{equation} \label{prop3entropy}
\intspacetime{\dfrac{\abs{u_{xx}}^2}{\Q{u_x}^3} + \delta \, \abs{u_{xx}}^2} \leq C(\approxu)
\end{equation}
holds true. Finally, the function \(u\) is continuous on \(\overline{\Omega}_T\) and satisfies
\begin{align}
\abs{u(t_0,x_0)} &\leq C_0(u_0) \cma \label{prop3holder0} \\
\abs{u(t_2, x_2) - u(t_1, x_1)} &\leq C_1(u_0,\delta) \, \bigl( \abs{t_2 - t_1}^\frac{1}{8} +  \abs{x_2 - x_1}^\frac{1}{2} \bigr) \label{prop3holder1}
\end{align}
for all \((t_i, x_i) \in \overline{\Omega}_T\), \(i = 0, 1, 2\). \qedthm
\end{prop}

\begin{proof}[Proof of Proposition \ref{proposition3}]
Let \((u^\epsilon, p^\epsilon)\) be a solution in the sense of Proposition \ref{proposition2}.

\beginstep{(1)}{A priori estimates.} Note that conservation of mass holds by the argument in \eqref{chap3boundux2}. By the energy estimate \eqref{chap3energyepsilon} we have the following estimates independent of \(\epsilon > 0\):
\begin{align}
\norm{u^\epsilon}{\bochner{L^\infty}{H^1}} &\leq C(u_0, \delta) \label{prf3apriori1} \\
\intspacetime{ m_{\epsilon}(u^\epsilon_x) \, \abs{p^\epsilon_x}^2} &\leq C(u_0) \label{prf3apriori2}
\end{align}
In particular, the second bound implies that
\begin{align}
\sqrt{\epsilon} \, \norm{p^\epsilon_x}{L^2(\Omega_T)} &\leq C(u_0) \cma \label{prf3boundepsilonpx} \\
\intspacetime{ m(u^\epsilon_x) \, \abs{p^\epsilon_x}^2} &\leq C(u_0) \fst \label{prf3boundmpx}
\end{align}
By \eqref{prf3apriori1} and the entropy estimate \eqref{chap3entropyepsilon} we have
\begin{gather}
\norm{u^\epsilon}{\bochner{L^2}{H^2}} \leq C(u_0, \delta) \fst \label{prf3apriori3}
\end{gather}
Moreover, for \(p^\epsilon\) we get
\begin{equation} \label{prf3estimatep}
\begin{split}
\norm{p^\epsilon}{L^2(\Omega_T)}^2 &\leq \intspacetime{\biggl\lvert \dfrac{u^\epsilon_{xx}}{\Q{u^\epsilon_x}^3} + \delta u^\epsilon_{xx} \biggr\rvert^2} \\
&\leq 2 \intspacetime{\dfrac{\abs{u^\epsilon_{xx}}^2}{\Q{u^\epsilon_x}^6} + \delta^2 \abs{u^\epsilon_{xx}}^2} \\
&\leq 2 \intspacetime{\dfrac{\abs{u^\epsilon_{xx}}^2}{\Q{u^\epsilon_x}^3} + \delta \abs{u^\epsilon_{xx}}^2} \cma
\end{split}
\end{equation}
such that
\begin{gather}
\norm{p^\epsilon}{L^2(\Omega_T)} \leq C(u_0) \label{prf3apriori4}
\end{gather}
by the entropy estimate \eqref{chap3entropyepsilon}. Furthermore, using
\[
\norml{m_\epsilon(u^\epsilon) \, p^\epsilon_x} \leq C(u_0)
\]
due to the energy estimate \eqref{chap3energyepsilon} and the boundedness of \(u^\epsilon\) (cf. Corollary \ref{corcont}), we show as before that
\begin{gather}
\norm{u^\epsilon_t}{\timespace} \leq C(u_0,m) \fst \label{prf3apriori5}
\end{gather}
Finally, by Corollary \ref{corcont} the function \(u^\epsilon\) is continuous on \(\overline{\Omega}_T\) and satisfies
\begin{align}
\abs{u^\epsilon(t_0,x_0)} &\leq C_0(u_0) \cma \label{prf3holder0} \\
\abs{u^\epsilon(t_2, x_2) - u^\epsilon(t_1, x_1)} &\leq C_1(u_0, \delta) \, \bigl( \abs{t_2 - t_1}^\frac{1}{8} +  \abs{x_2 - x_1}^\frac{1}{2} \bigr) \label{prf3holder1}
\end{align}
for all \((t_i, x_i) \in \overline{\Omega}_T\), \(i = 0, 1, 2\).

\beginstep{(2)}{The limit \(\epsilon \limeps 0\).}
By Arzel\`{a}--Ascoli's theorem, there exists a subsequence, again denoted by \(\sequence{u^\epsilon}{\epsilon > 0}\), such that
\begin{alignat}{2}
u^\epsilon &\strongto u &&\text{ uniformly on } \overline{\Omega}_T \label{prf3uniconv}
\end{alignat}
as \(\epsilon \limeps 0\). Since the set of functions satisfying the estimates \eqref{prf3holder0} and \eqref{prf3holder1} is closed with respect to uniform convergence, \(u\) is continuous and satisfies \eqref{prop3holder0} and \eqref{prop3holder1} in the limit.

Using the a priori estimates \eqref{prf3apriori1}, \eqref{prf3apriori3}, \eqref{prf3apriori4}, \eqref{prf3apriori5} and Simon's theorem \ref{thmsimon}, we infer the following convergence results for a further subsequence, again denoted by \(\sequence{u^\epsilon}{\epsilon > 0}\) and \(\sequence{p^\epsilon}{\epsilon > 0}\), for \(\epsilon \limeps 0\):
\begin{alignat}{3}
u^\epsilon &\weakto u &&\weaktext &&\bochner{L^2}{H^2} \label{prf3weakconvu}\\
u^\epsilon &\weakstarto u &&\weakstartext &&\bochner{L^\infty}{H^1} \\
u^\epsilon &\strongto u &&\strongtext &&\bochner{L^2}{H^1} \label{prf3strconvu} \\
u^\epsilon &\strongto u &&\strongtext &&\bochnerex{C}{[0,T]}{L^2(\Omega)} \\
u^\epsilon_t &\weakto u_t &&\weaktext &&\timespace \\
p^\epsilon &\weakto p &&\weaktext &&L^2(\Omega_T) \label{prf3weakconvp}
\end{alignat}
This proves the regularity results \eqref{prop3regu} and \eqref{prop3regut}. \eqref{prop3boundary} and \eqref{prop3initial} follow from the same arguments as before.

Again, from the strong convergence of \(u^\epsilon_x\) in \(L^2(\Omega_T)\) \eqref{prf3strconvu} we deduce that
\begin{alignat}{3}
u^\epsilon_x &\strongto u_x &&\pwaetext &&\Omega_T \label{prf3uxpwae}
\end{alignat}
for a subsequence and hence argue as in the proof of Proposition \ref{proposition2} to infer that
\begin{alignat*}{3}
- \ddx \Qex{u^\epsilon_x} &\weakto - \ddx \Qex{u_x} &&\weaktext &&L^2(\Omega_T) \fst
\end{alignat*}
This proves identity \eqref{prop3sysb}.

Let us now show that indeed
\begin{equation} \label{prf3weakconvpx}
p^\epsilon_x \weakto p_x \weaktext L^2(U) \qquad \forall \, U \subset \subset \possettimeabs{u} \fst
\end{equation}
To this end, let \(\possubset \subset\subset \possettimeabs{u}\) be arbitrary and denote
\[
\posparam := \min_{(t,x) \in \overline{\possubset}} \abs{u(t,x)} > 0 \fst
\]
By the continuity of \(u\), we have \(\possubset \subset \possubsett := \possetex{\abs{u}}{\frac{\posparam}{2}}_T\), and by the uniform convergence \eqref{prf3uniconv} it holds that \(u^\epsilon(t,x) \geq \frac{\posparam}{4}\) on \(\possubsett\) for \(\epsilon > 0\) sufficiently small. Thus, by the energy estimate \eqref{chap3energyepsilon},
\[
c_{\mobextra} \left( \frac{\posparam}{4} \right)^n \intspacetimeex{\possubsett}{\abs{p^\epsilon_x}^2} \leq \intspacetime{m_{\epsilon}(u^\epsilon) \, \abs{p^\epsilon_x}^2} \leq C(u_0) \cma
\]
which implies
\[
\intspacetimeex{\possubsett}{\abs{p^\epsilon_x}^2} \leq C(u_0, c_{\mobextra}, \possubsett) \fst
\]
Hence, there exists a subsequence depending on \(\possubsett\) such that
\begin{alignat}{3}
p^\epsilon_x &\weakto q &&\weaktext &&L^2(\possubsett) \fst \label{prf3weakconvpxprev}
\end{alignat}
Since \(u\) is continuous on \(\overline{\Omega}_T\), there exists a finite number of rectangles \(R_1, R_2, \dotsc , R_k\) such that \(\possubset \subset \bigcup_{j = 1}^k R_j \subset \possubsett\). On each \(R_j\), \(j = 1, \dotsc, k\), we choose an arbitrary \(\varphi \in C^\infty_0(R_j)\) and deduce from \eqref{prf3weakconvp} and \eqref{prf3weakconvpxprev} that the relation
\[
\intspacetimeex{R_j}{q \cdot \varphi} = -\intspacetimeex{R_j}{p \cdot \varphi_x}
\]
holds true. Hence, we may identify \(q = p_x\) on \(\possubset\) and deduce that in fact the whole sequence converges, that is,
\[
p^\epsilon_x \weakto p_x \text{ weakly in } L^2(\possubset)
\]
as \(\epsilon \limeps 0\). The arbitrariness of \(\possubset\) proves \eqref{prf3weakconvpx} and in turn \eqref{prop3regp}.

We shall now prove that
\begin{equation}
\intspacetime{m_\epsilon(u^\epsilon) \, p^\epsilon_x \cdot v_x} \xrightarrow{\epsilon \limeps 0} \intspacetimeex{\possettimeabs{u}}{m(u) \, p_x \cdot v_x} \fst \label{prf3convmpint}
\end{equation}
for arbitrary \(v \in \bochner{L^2}{H^1}\). First of all, observe that
\[
\intspacetime{m_\epsilon(u^\epsilon) \, p^\epsilon_x \cdot v_x} = \intspacetime{m(u^\epsilon) \, p^\epsilon_x \cdot v_x} + \epsilon \intspacetime{p^\epsilon_x \cdot v_x} \cma
\]
and for the second term we have by \eqref{prf3boundepsilonpx}
\[
\epsilon \intspacetime{p^\epsilon_x \cdot v_x} \leq \epsilon \norm{p^\epsilon_x}{L^2(\Omega_T)} \norm{v_x}{L^2(\Omega_T)} \xrightarrow{\epsilon \limeps 0} 0 \fst
\]
Moreover, for arbitrary \(\rho > 0\) we split
\[
\begin{split}
\intspacetime{m(u^\epsilon) \, p^\epsilon_x \cdot v_x} &= \intspacetimeex{\possetex{\abs{u}}{\rho}_T}{m(u^\epsilon) \, p^\epsilon_x \cdot v_x} + \intspacetimeex{\setex{\abs{u} \leq \rho}_T}{m(u^\epsilon) \, p^\epsilon_x \cdot v_x} \\
&=: \mathcal{I}_1 + \mathcal{I}_2 \fst
\end{split}
\]
We deduce from the weak convergence of \(p^\epsilon_x\) \eqref{prf3weakconvpx}, the uniform convergence of \(u^\epsilon\) \eqref{prf3uniconv} and the boundedness of \(m(u^\epsilon)\) \eqref{prf3holder0}, that
\begin{equation} \label{prf3weakI1}
\mathcal{I}_1 \xrightarrow{\epsilon \limeps 0} \intspacetimeex{\possetex{\abs{u}}{\rho}_T}{m(u) \, p_x \cdot v_x} \fst
\end{equation}
On the other hand, using \eqref{prf3boundmpx} we may estimate
\begin{equation} \label{prf3boundmprho}
\begin{split}
\abs{\mathcal{I}_2} &= \Bigl\vert \intspacetimeex{\setex{\abs{u} \leq \rho}_T}{m(u^\epsilon) \, p^\epsilon_x \cdot v_x} \Bigr\vert \\
&\leq \Bigl( \intspacetimeex{\setex{\abs{u} \leq \rho}_T}{m(u^\epsilon) \, \abs{v_x}^2} \Bigr)^\frac{1}{2} \Bigl( \intspacetimeex{\setex{\abs{u} \leq \rho}_T}{m(u^\epsilon) \, \abs{p^\epsilon_x}^2} \Bigr)^\frac{1}{2} \\
&\leq \rho^\frac{n}{2} \cdot C(m_0) \, \norm{v}{\bochner{L^2}{H^1}} \, \Bigl( \intspacetimeex{\setex{\abs{u} \leq \rho}_T}{m(u^\epsilon) \, \abs{p^\epsilon_x}^2} \Bigr)^\frac{1}{2} \\
&\leq C(u_0, m_0, v) \, \rho^\frac{n}{2}
\end{split}
\end{equation}
independently of \(\epsilon > 0\). Hence, the arbitrariness of \(\rho > 0\) proves \eqref{prf3convmpint}, and equation \eqref{prop3sysa} follows.

It remains to prove the estimates \eqref{prop3energy} and \eqref{prop3entropy}. As before, we obtain the boundedness of the first and second term in the energy estimate \eqref{prop3energy} by the pointwise convergence of \(u^\epsilon_x\) \eqref{prf3uxpwae} and Fatou's lemma.

In view of the proof of Proposition \ref{proposition2}, we would like to prove that
\begin{alignat}{3} \label{prf3nowlsc}
m(u^\epsilon)^\frac{1}{2} \, p^\epsilon_x \weakto m(u)^\frac{1}{2} \, p_x \, \chi_{\possettimeabs{u}} &&\weaktext &&L^2(\Omega_T)
\end{alignat}
and infer that
\[
\intspacetimeex{\possettimeabs{u}}{m(u) \, \abs{p_x}^2} \leq C(u_0)
\]
by weak lower semicontinuity of the norm. However, note that we are not able to prove \eqref{prf3nowlsc} by the arguments above, namely, we would not succeed in estimating \(\mathcal{I}_2\) in terms of \(\rho\) as in \eqref{prf3boundmprho}, if we would replace \(m(u^\epsilon)\) by \(m(u^\epsilon)^\frac{1}{2}\). But, for any \(r > 1\) and arbitrary \(w \in L^2(\Omega_T)\), we still may estimate
\begin{equation*}
\begin{split}
&\Bigl\vert \intspacetimeex{\setex{\abs{u} \leq \rho}_T}{m(u^\epsilon)^\frac{r}{2} \, p^\epsilon_x \cdot w} \Bigr\vert \\
&\leq \Bigl( \intspacetimeex{\setex{\abs{u} \leq \rho}_T}{m(u^\epsilon)^{r-1} \, \abs{w}^2} \Bigr)^\frac{1}{2} \Bigl( \intspacetimeex{\setex{\abs{u} \leq \rho}_T}{m(u^\epsilon) \, \abs{p^\epsilon_x}^2} \Bigr)^\frac{1}{2} \\
&\leq \rho^{\frac{(r-1)n}{2}} \cdot C(m_0) \, \norm{w}{L^2(\Omega_T)} \, \Bigl( \intspacetimeex{\setex{\abs{u} \leq \rho}_T}{m(u^\epsilon) \, \abs{p^\epsilon_x}^2} \Bigr)^\frac{1}{2} \\
&\leq C(u_0, m_0, w) \, \rho^{\frac{(r-1)n}{2}} \fst
\end{split}
\end{equation*}
In addition, it holds by the same arguments as in \eqref{prf3weakI1} that
\[
\intspacetimeex{\possetex{\abs{u}}{\rho}_T}{m(u^\epsilon)^\frac{r}{2} \, p^\epsilon_x \cdot w} \xrightarrow{\epsilon \limeps 0} \intspacetimeex{\possetex{\abs{u}}{\rho}_T}{m(u)^\frac{r}{2} \, p_x \cdot w} \fst
\]
Hence, we deduce that
\begin{alignat*}{3}
m(u^\epsilon)^\frac{r}{2} \, p^\epsilon_x &\weakto m(u)^\frac{r}{2} \, p_x \, \chi_{\possettimeabs{u}} &&\weaktext &&L^2(\Omega_T)
\end{alignat*}
for any \(r > 1\). Moreover,
\[
\norm{m(u^\epsilon)^\frac{r}{2} \, p^\epsilon_x}{L^2(\Omega_T)}^2 \leq \norm{m(u^\epsilon)^{r-1}}{L^\infty(\Omega_T)} \intspacetime{m(u^\epsilon) \, \abs{p^\epsilon_x}^2} \leq C(u_0, r)
\]
by \eqref{prf3boundmpx} and \eqref{prf3holder0}. Therefore, weak lower semicontinuity of the norm proves the energy estimate \eqref{prop3energy}.

Finally, to prove \eqref{prop3entropy}, we observe once more that by \eqref{prf3uxpwae} \(u^\epsilon_x\) converges pointwise almost everywhere and that \(\Q{u^\epsilon_x}^{-3/2} \leq 1\) holds. This together with Vitali's convergence theorem and the weak convergence of \(u^\epsilon_{xx}\) \eqref{prf3weakconvu} yields
\begin{alignat*}{3}
\dfrac{u^\epsilon_{xx}}{\Q{u^\epsilon_x}^{\frac{3}{2}}} &\weakto \dfrac{u_{xx}}{\Q{u_x}^{\frac{3}{2}}} &&\weaktext &&L^2(\Omega_T) \fst
\end{alignat*}
by the arguments above. Thus, \eqref{prop3entropy} follows from weak lower semicontinuity of the norm and the entropy estimate \eqref{chap3entropyepsilon}.
\end{proof}

\begin{rem} \label{remarkenergyr}
Note that \eqref{prf3convmpint} would directly imply
\[
\intspacetimeex{\possettimeabs{u}}{m(u)^2 \, \abs{p_x}^2} \leq C(u_0) \fst
\]
However, it will turn out that the effort we spent on proving
\[
\intspacetimeex{\possettimeabs{u}}{m(u)^r \, \abs{p_x}^2} \leq C(u_0, r)
\]
for \(r > 1\) is necessary. In fact, we will need the energy estimate \eqref{prop3energy} for some \(1 < r < 2\) in the proof of Theorem \ref{theorem1} in order to identify the flux \(m(u) \, p_x\) by an analogous argumentation as in the preceding proof.
\qedrem
\end{rem}

\subsection{Nonnegativity}

We may indeed prove that, given \(n \geq 1\), any solution in the sense of Proposition \ref{proposition3} is nonnegative. 
Again, this result was introduced by Bernis and Friedman (see \cite[Theorem 4.1]{berfri}) in the setting of the standard thin-film equation. Their proof, which relies on the entropy estimate and the continuity of the solution, can be literally applied to our case.

For the upcoming results, we define
\begin{alignat*}{3}
g(s) &:= \lim\limits_{\epsilon \limeps 0} g_\epsilon(s) \qquad &\forall& \, s \in \IR \cma \\
G(s) &:= \lim\limits_{\epsilon \limeps 0} G_\epsilon(s) \qquad &\forall& \, s \in \IR \fst
\end{alignat*}

We are going to prove the following proposition:

\begin{prop} \label{proposition3pos}
Assume (H1)--(H2), let \(0 < \delta \leq 1\) and \(u = u^\delta\) be a solution in the sense of Proposition \ref{proposition3}.
\begin{enumerate}
\item[(i)] If \(n \geq 1\), then
\begin{equation} \label{prop3posnn}
u \geq 0
\end{equation}
holds on \(\overline{\Omega}_T\).
\item[(ii)] If \(n \geq 2\), we have
\begin{equation} \label{prop3boundG}
\intspace{G(u(t))} \leq C(u_0)
\end{equation}
for all \(t \in [0,T]\). \qedthm
\end{enumerate}
\end{prop}

\begin{proof}
%

As in \cite[Section 4]{berfri}, one may calculate the following growth properties for the function \(G(s) = \lim_{\epsilon \limeps 0} G_\epsilon(s)\) for \(0 < s \leq a\) (remember the choice of \(a\) in \eqref{chap3defa}):
\begin{equation} \label{chap3growthG}
G(s) =
\begin{cases}
c_0 + \mathcal{O}(s^{2-n}) &\text{ if } 1 < n < 2 \\
c_1 \log{\frac{1}{s}} + \mathcal{O}(1) &\text{ if } n = 2 \\
c_2 s^{2-n} + R(s) &\text{ if } n > 2
\end{cases}
\end{equation}
Here, the positive constants \(c_0\), \(c_1\) and \(c_2\) depend on \(\mobextra\) only and we have
\[
R(s) =
\begin{cases}
\mathcal{O}(1) &\text{ if } 2 < n < 3 \cma \\
\mathcal{O}(\log{\frac{1}{s}}) &\text{ if } n = 3 \cma \\
\mathcal{O}(s^{2-n}) &\text{ if } n > 3 \fst
\end{cases}
\]

Now, let \(u\) be a solution in the sense of Proposition \ref{proposition3} and \(\sequence{u^\epsilon}{\epsilon > 0}\) the sequence that we have obtained in the course of the preceding proof. In particular, we may suppose that
\begin{alignat}{2}
u^\epsilon &\strongto u &&\text{ uniformly on } \overline{\Omega}_T \label{prf3posuniconv}
\end{alignat}
as \(\epsilon \to 0\), and the functions \(\sequence{u^\epsilon}{\epsilon > 0}\) satisfy the entropy estimate \eqref{chap3entropyepsilon}.

\beginstep{(1)}{ad (i).} We argue by contradiction and assume that there exists a point \((t_0, x_0) \in \overline{\Omega}_T\) such that \(u(t_0, x_0) < 0\). Due to \eqref{prf3posuniconv}, there is \(\sigma > 0\) such that
\[
u^\epsilon(t_0, x) < -\sigma
\]
for \(x \in \overline{\Omega}\), \(\abs{x - x_0} < \sigma\) and \(\epsilon > 0\) sufficiently small. For such \(x\) we obtain
\[
G_\epsilon(u^\epsilon(t_0, x)) = - \int\limits_{u^\epsilon(t_0, x)}^a {g_\epsilon(s) \, ds} \geq - \int\limits_{-\sigma}^0 {g_\epsilon(s) \, ds} \xrightarrow{\epsilon \limeps 0} - \int\limits_{-\sigma}^0 {g(s) \, ds}
\]
by the monotone convergence theorem.
However, since \(g(s) = -\infty\) for \(s < 0\) and \(n \geq 1\), the integral on the right-hand side is equal to \(+\infty\). Hence,
\[
\liminf\limits_{\epsilon \limeps 0} \intspace{G_\epsilon(u^\epsilon(t_0, x))} = + \infty
\]
in contradiction to the entropy estimate \eqref{chap3entropyepsilon}.

\beginstep{(1)}{ad (ii).} Let now \(n \geq 2\). We shall at first prove that
\begin{equation} \label{chap3muE}
\mu \left(\zeroset{u(t)}\right) = 0
\end{equation}
for all \(t \in [0,T]\). To this end, we argue again by contradiction and assume that there exists \(t_0 \in [0,T]\) such that \(\mu(E) > 0\), where \(E := \zeroset{u(t_0)} \subseteq \overline{\Omega}\). Let \(\rho > 0\). In view of \eqref{prf3posuniconv}, it holds that
\[
u^\epsilon(t_0, x) < \rho
\]
for \(x \in E\) and \(\epsilon > 0\) sufficiently small. Similarly as before, it follows from the monotone convergence theorem that
\[
 G_\epsilon(u^\epsilon(t_0, x)) = - \int\limits_{u^\epsilon(t_0, x)}^a {g_\epsilon(s) \, ds} \geq - \int\limits_{\rho}^{a}{g_\epsilon(s) \, ds} \xrightarrow{\epsilon \limeps 0} -\int\limits_{\rho}^{a}{g(s) \, ds} \fst
\]
and by \eqref{chap3growthG} we have
\[
-\int\limits_{\rho}^{a}{g(s) \, ds} = G(\rho) \geq
\begin{cases}
c \, \log \frac{1}{\rho} &\text{ if } n = 2 \cma \\
c \, \rho^{2-n} &\text{ if } n > 2 \cma
\end{cases}
\]
for a positive constant \(c > 0\) and \(\rho > 0\) sufficiently small. Therefore,
\[
\liminf\limits_{\epsilon \limeps 0} \intspace{G_\epsilon(u^\epsilon(t_0, x))} \geq
\begin{cases}
c \, \log \frac{1}{\rho} \, \mu(E) &\text{ if } n = 2 \cma \\
c \, \rho^{2-n} \, \mu(E) &\text{ if } n > 2 \fst
\end{cases}
\]
Letting \(\rho \limeps 0\), we obtain a contradiction to the entropy estimate \eqref{chap3entropyepsilon}, and \eqref{chap3muE} follows. Finally, we have for every point \((t,x)\in [0,T] \times \Omega\) where \(u(t,x) > 0\)
\[
G_\epsilon(u^\epsilon(t,x)) \xrightarrow{\epsilon \limeps 0} G(u(t,x)) \fst
\]
Hence, using \eqref{chap3muE}, we deduce from the entropy estimate \eqref{chap3entropyepsilon} and Fatou's lemma that
\[
\intspace{G(u(t, x))} \leq C(u_0)
\]
for all \(t \in [0,T]\), and the proof is complete. \qedhere
\end{proof}

\begin{rem}
Observe that \(\mu \left(\zeroset{u}_T \right) = 0\) holds true if \(n \geq 2\), which we have proved above as an intermediate result. Following the lines of \cite[Theorem 4.1]{berfri}, we could moreover show that \(u > 0\) on \(\overline{\Omega}_T\) if \(n \geq 4\). However, note that these results would be worthless in view of the upcoming limit process \(\delta \limeps 0\), even if the approximating solutions \(\sequence{u^\delta}{\delta > 0}\) would converge uniformly on \(\overline{\Omega}_T\). \qedrem
\end{rem}

\begin{rem}
Since the solution \(u\) in the sense of Proposition \ref{proposition3} is nonnegative because of Proposition \ref{proposition3pos} and since we assume \(n \geq 1\) throughout this work, we may replace
\[
\possettimeabs{u} = \possettime{u}
\]
and similar expressions on every occurrence. \qedrem
\end{rem}

\clearsection

\section{Proofs of the Theorems} \label{secfinal}

In the previous step, we obtained a solution \((u^\delta, p^\delta)\), \(0 < \delta < 1\), of the following system:
\begin{subequations} \label{chap4sysdelta}
\begin{align}
\inttime{\pairh{u^\delta_t}{v}} &+ \intspacetimeex{\possettime{u^\delta}}{m(u^\delta) \, p^\delta_x \cdot v_x} = 0 &\forall& \, v \in \bochner{L^2}{H^1} \label{chap4sysdeltaa} \\
p^\delta &= - \ddx \leftA \Qex{u^\delta_x} + \delta u^\delta_x \rightA \label{chap4sysdeltab}
\end{align}
\end{subequations}
In particular, note that the solution is continuous and nonnegative on \(\overline{\Omega}_T\) by Proposition \ref{proposition3pos} (i).
Moreover, the solution satisfies the energy estimate
\begin{equation} \label{chap4energydelta}
\begin{split}
\intspace{\Q{u^\delta_x(t)}} + \frac{\delta}{2} \norml{u^\delta_x(t)}^2 + \intspacetimeex{\possettime{u^\delta}}{m(u^\delta)^r \, \abs{p^\delta_x}^2} \leq C(\approxu, r)
\end{split}
\end{equation}
for almost all \(t \in (0,T)\) and arbitrary \(r > 1\), and
\begin{equation} \label{chap4entropydelta}
\intspacetime{\dfrac{\abs{u^\delta_{xx}}^2}{\Q{u^\delta_x}^3} + \delta \abs{u^\delta_{xx}}^2} \leq C(u_0)
\end{equation}
holds. If in addition \(n \geq 2\), we have
\begin{equation} \label{chap4boundGdelta}
\intspace{G(u^\delta(t))} \leq C(u_0)
\end{equation}
for all \(t \in [0,T]\) by Proposition \ref{proposition3pos} (ii).

In the last step, we intend to let \(\delta \limeps 0\) and to prove Theorems \ref{theorem1} and \ref{theorem2}. In view of the estimates \eqref{chap4energydelta} and \eqref{chap4entropydelta}, we will lose control of the \(\bochner{L^\infty}{H^1}\)- and \(\bochner{L^2}{H^2}\)-norm of \(\sequence{u^\delta}{\delta > 0}\). In other words, the operator \(A_\delta\) is not uniformly coercive for \(\delta \limeps 0\). Furthermore, the linear growth of the function \(s \mapsto \Q{s}\) is not sufficient to deduce any compactness results for a sequence of functions \(\sequence{u^\delta}{\delta > 0} \subset \bochner{L^\infty}{W^{1,1}}\) satisfying the energy estimate \eqref{chap4energydelta}. Of course, this is a well-known issue of non-parametric mean curvature type equations. However, due to the degenerate and higher-order structure of our equation, the usual approaches known from this theory are not applicable to our problem. (For an overview of methods for mean curvature type equations, we refer the reader to \cite{giltru}.)

Instead, we hope that the uniform bounds on \(\sequence{u^\delta}{\delta > 0}\) induced by the estimates \eqref{chap4energydelta} and \eqref{chap4entropydelta}, namely
\[
\esssup_{t \in (0,T)} \intex{\Omega}{ }{\Q{u^\delta_x(t)}}{dx} \leq C(u_0) \cma \qquad \intspacetime{\dfrac{\abs{u^\delta_{xx}}^2}{\Q{u^\delta_x}^3}} \leq C(u_0) \cma
\]
are strong enough to ensure an uniform bound on \(\sequence{u^\delta_x}{\delta > 0}\) or even \(\sequence{u^\delta_{xx}}{\delta > 0}\) in a reflexive space. This does not seem to be entirely true either. But, supposing that
\[
\intspace{\Q{u^\delta_x(t)}} \leq C(u_0, t), \qquad \intspace{\dfrac{\abs{u^\delta_{xx}(t)}^2}{\Q{u^\delta_x(t)}^3}} \leq C(u_0, t) 
\]
hold true for \(t \in (0,T)\), we are able to show that
\[
\norm{u^\delta(t)}{H^2(\Omega)} \leq C(u_0, t)
\]
independently of \(\delta > 0\). This is the basic idea of the following approach and the results in Theorem \ref{theorem1}. That is, we will obtain very low regularity results for the solutions \(u\) and \(p\) as functions on \(\Omega_T\), but for almost all \(t \in (0,T)\) the functions \(u(t)\) and \(p(t)\) will possess more regular representations as functions on \(\Omega\).

Before we prove Theorems \ref{theorem1} and \ref{theorem2}, we are going to derive several helpful results in the following section.

\subsection{Preliminaries}

The first lemma gives a notion of estimates pointwise in time for sequences of functions that are uniformly bounded in \(\bochnerex{L^1}{I}{\IR}\). The very same result was proved in \cite[Theorem 8.2]{gruen2}.

\begin{lem} \label{lemunibound}
Let \(\sequence{f_k}{k \in \IN} \subset \bochnerex{L^1}{I}{\IR}\) such that \(f_k \geq 0\) for \(k \in \IN\) and
\begin{equation} \label{lemuniboundcond}
\norm{f_k}{\bochnerex{L^1}{I}{\IR}} \leq C \cma
\end{equation}
where \(C > 0\) is a constant independent of \(k \in \IN\). Then there exists a subsequence \(\sequence{f_{k_j}}{j \in \IN}\) such that
\[
\abs{f_{k_j}(t)} \leq C(t) < \infty
\]
for \(j \in \IN\) and almost all \(t \in (0,T)\). \qedthm
\end{lem}

\begin{proof}
Consider the set
\[
E := \{ t \in (0,T) : \liminf_{k \to \infty} f_k(t) = +\infty \}
\]
and assume that
\begin{equation} \label{prfuniboundblitz}
\mu(E) > 0 \fst
\end{equation}
For arbitrary \(L > 0\), we have by Fatou's lemma and \eqref{lemuniboundcond}
\[
\begin{split}
L \cdot \mu(E) &= \int\limits_E {\liminf\limits_{k \to \infty} \, \min \lbrace f_k(t), L \rbrace \, dt} \leq \liminf\limits_{k \to \infty} \int\limits_E {\min \lbrace f_k(t), L \rbrace \, dt} \\
&\leq \liminf\limits_{k \to \infty} \int\limits_E {f_k(t) \, dt} \leq C \fst
\end{split}
\]
Hence, letting \(L \to +\infty\) contradicts \eqref{prfuniboundblitz} and the lemma is proved.
\end{proof}

Given the estimates \eqref{chap4energydelta} and \eqref{chap4entropydelta} and using the previous Lemma \ref{lemunibound}, we are able to prove an uniform \(H^2(\Omega)\)-bound on \(\sequence{u^\delta(t)}{\delta > 0}\) for almost all \(t \in (0,T)\). This will be the most crucial result in this work, as otherwise no identification of the pressure \(p\) and of the flux \(m(u) \, p_x\) would be possible. Since we will consider the sequence \(\sequence{u^\delta(t)}{\delta > 0} \subset H^2(\Omega)\) for arbitrary but fixed \(t \in (0,T)\), we will drop the dependency on \(t\) for shorthand notation in the following lemma.

\begin{lem} \label{lemcrazy}
Let \(\sequence{u^\delta}{\delta > 0} \subset H^2(\Omega)\) satisfy the estimates
\begin{gather}
\intspace{u^\delta} \leq c_0 \cma \label{lemcrazyc0} \\
\intspace{\Q{u^\delta_x}} \leq c_1 \cma \label{lemcrazyc1} \\
\intspace{\dfrac{\abs{u^\delta_{xx}}^2}{\Q{u^\delta_x}^3}} \leq c_2 \fst \label{lemcrazyc2}
\end{gather}
Then \(\sequence{u^\delta}{\delta > 0}\) is uniformly bounded in \(W^{1, \infty}(\Omega) \cap H^2(\Omega)\) independently of \(\delta > 0\). \qedthm
\end{lem}

\begin{proof}
We introduce the functions
\begin{align}
Q^\delta(x) &:= \Q{u^\delta_x(x)} \cma \\
f^\delta(x) &:= \dfrac{u^\delta_x(x)}{Q^\delta(x)} \cma \label{prfcrazydeff} \\
g^\delta(x) &:= (1 - \abs{f^\delta(x)}^2)^\frac{1}{4} = Q^\delta(x)^{-\frac{1}{2}} \cma \label{prfcrazydefg}
\end{align}
where \(x \in \overline{\Omega}\). Note that by assumption \(u^\delta \in H^2(\Omega) \embedding C^{1, \beta}(\overline{\Omega})\) for \(0 \leq \beta \leq \frac{1}{2}\) such that we may assume \(f^\delta\) to be continuous on \(\overline{\Omega}\). We set
\begin{equation}
y_\delta := \max_{x \in \overline{\Omega}}{\abs{f^\delta(x)}} \in [0, 1) \label{prfcrazydefy}
\end{equation}
and claim that \(y_\delta\) is bounded by a constant strictly smaller than \(1\) independently of \(\delta > 0\). We calculate
\[
g^\delta_x(x) = \ddx Q^\delta(x)^{-\frac{1}{2}} = - \frac{1}{2} Q^\delta(x)^{-\frac{3}{2}} \cdot \dfrac{u^\delta_x(x)}{Q^\delta(x)} \cdot u^\delta_{xx}(x) 
= - \frac{1}{2} f^\delta(x) \frac{u^\delta_{xx}(x)}{Q^\delta(x)^{\frac{3}{2}}} \fst
\]
Hence, by \eqref{lemcrazyc2} and H\"{o}lder's inequality,
\[
\begin{split}
\intspace{\abs{g^\delta_x(x)}^2} &= \frac{1}{4} \intspace{\abs{f^\delta(x)}^2 \frac{\abs{u^\delta_{xx}(x)}^2}{Q^\delta(x)^3}} \\
& \leq \frac{1}{4} \norm{f^\delta}{L^\infty(\Omega)}^2 \intspace{\frac{\abs{u^\delta_{xx}(x)}^2}{Q^\delta(x)^3}} \leq \frac{1}{4} \, c_2 \fst
\end{split}
\]
Since moreover
\[
\norm{g^\delta}{L^\infty(\Omega)} = \norm{(Q^\delta)^{-\frac{1}{2}}}{L^\infty(\Omega)} \leq 1 \cma
\]
it follows that \(g^\delta \in H^1(\Omega) \hookrightarrow C^{0,\frac{1}{2}}(\overline{\Omega})\) and in particular
\begin{equation} \label{prfcrazygholder}
\norm{g^\delta}{C^{0,\frac{1}{2}}(\overline{\Omega})} \leq C \, \norm{g^\delta}{H^1(\Omega)} \leq C(c_2) =: K \cma
\end{equation}
where the constant \(K\) is independent of \(\delta > 0\). We now choose \(x_\delta \in \overline{\Omega}\) such that \(\abs{f^\delta(x_\delta)} = y_\delta\). Then, \eqref{prfcrazydefg} and \eqref{prfcrazygholder} imply
\[
\bigl\lvert g^\delta(x) - \bigl(1 - y_\delta^2\bigr)^\frac{1}{4} \bigr\rvert \leq K \, \abs{x - x_\delta}^\frac{1}{2}
\]
for all \(x \in \overline{\Omega}\). Since \(g^\delta(x) \geq \bigl(1 - y_\delta^2\bigr)^\frac{1}{4} > 0\) by \eqref{prfcrazydefg} and \eqref{prfcrazydefy}, we obtain
\[
g^\delta(x) \leq K \, \abs{x - x_\delta}^\frac{1}{2} + \bigl(1 - y_\delta^2\bigr)^\frac{1}{4}
\]
and may moreover estimate
\[
g^\delta(x)^2 \leq 2K^2 \, \abs{x - x_\delta} + 2 \, \bigl(1 - y_\delta^2\bigr)^\frac{1}{2}
\]
for all \(x \in \overline{\Omega}\). From this we deduce that
\[
\frac{1}{2} \intspace{\dfrac{1}{K^2 \, \abs{x - x_\delta} + \bigl(1 - y_\delta^2\bigr)^\frac{1}{2}}} \leq \intspace{\dfrac{1}{g^\delta(x)^2}} \fst
\]
Having in mind that we have chosen \(\Omega = (-l, l)\) without loss of generality, we may estimate the left-hand side from below by
\[
\frac{1}{2} \intspace{\dfrac{1}{K^2 \, \abs{x - l} + \bigl(1 - y_\delta^2\bigr)^\frac{1}{2}}} \leq
\frac{1}{2} \intspace{\dfrac{1}{K^2 \, \abs{x - x_\delta} + \bigl(1 - y_\delta^2\bigr)^\frac{1}{2}}} \cma
\]
whereas we get for the right-hand side by \eqref{lemcrazyc1}
\[
\intspace{\dfrac{1}{g^\delta(x)^2}} = \intspace{Q^\delta(x)} \leq c_1 \fst
\]
Thus, we finally arrive at
\begin{equation} \label{prfcrazyint}
\frac{1}{2} \intspace{\dfrac{1}{K^2 \, \abs{x - l} + \bigl(1 - y_\delta^2\bigr)^\frac{1}{2}}} \leq c_1 \fst
\end{equation}
Now, we observe that the left-hand side of \eqref{prfcrazyint} depends continuously on \(y_\delta\) and tends to \(+\infty\) if we let \(y_\delta \to 1\). Hence, we immediately infer the existence of a constant \(0 < M = M(c_1, c_2) < 1\) such that
\[
\abs{f^\delta(x)} \leq y_\delta \leq M
\]
for all \(x \in \overline{\Omega}\) and independently of \(\delta > 0\). By the definition of \(f^\delta\) in \eqref{prfcrazydeff}, this implies
\[
\abs{u^\delta_x(x)} \leq M \cdot \Q{u^\delta_x(x)}
\]
and therefore
\[
\abs{u^\delta_x(x)} \leq \dfrac{M}{\sqrt{1-M^2}} < \infty \fst
\]
This together with assumption \eqref{lemcrazyc0} and Poincar\'e's inequality proves
\[
\norm{u^\delta}{W^{1,\infty}(\Omega)} \leq C(c_0, c_1, c_2)
\]
independently of \(\delta > 0\). As a consequence,
\[
\begin{split}
\norm{u^\delta_{xx}}{L^2(\Omega)}^2 &= \intspace{Q^\delta(x)^3 \dfrac{\abs{u^\delta_{xx}(x)}^2}{Q^\delta(x)^3}} \\
&\leq \norm{(Q^\delta)^3}{L^\infty(\Omega)} \intspace{\dfrac{\abs{u^\delta_{xx}}^2}{Q^\delta(x)^3}} \\
&\leq C(c_0, c_1, c_2) 
\end{split}
\]
and the proof is complete. \qedhere
\end{proof}

\begin{rem} \label{remarkcrazy1} 
Note that the conditions \eqref{lemcrazyc0} and \eqref{lemcrazyc1} in Lemma \ref{lemcrazy} are immediately satisfied for the functions \(\sequence{u^\delta(t)}{\delta > 0} \subset H^2(\Omega)\) for almost all \(t \in (0,T)\) by conservation of mass and the energy estimate \eqref{chap4energydelta}, whereas condition \eqref{lemcrazyc2} will be deduced from the entropy estimate \eqref{chap4entropydelta} and Lemma \ref{lemunibound}. Of course, observing that
\[
\intspace{\dfrac{\abs{u^\delta_{xx}(t)}^2}{\Q{u^\delta_x(t)}^3}} \leq \varphi(t) , \quad \norm{\varphi}{\bochnerex{L^1}{I}{\IR}} \leq C(u_0) < \infty
\]
by the entropy estimate \eqref{chap4entropydelta}, one might hope that the proof of Lemma \ref{lemcrazy} can be adopted to obtain an uniform \(\bochner{L^p}{H^2}\)-bound on \(\sequence{u^\delta}{\delta > 0}\) for some \(1 < p \leq \infty\): Quantifying the constant \(K\) in the preceding proof in terms of \(\varphi(t)\), one is able calculate the time-dependent constant \(M(\varphi(t))\) from \eqref{prfcrazyint} such that \(y_\delta(t) \leq M(\varphi(t)) < 1\) holds for almost all \(t \in (0,T)\). However, by this procedure we eventually end up with
\[
\norm{u^\delta_x(t)}{L^\infty(\Omega)} \leq \dfrac{M(\varphi(t))}{\sqrt{1-M(\varphi(t))^2}} \approx C \cdot \exp(\varphi(t)^2) \cdot \varphi(t)^{-2} \cma
\]
such that not necessarily \(\norm{u^\delta_x(t)}{L^\infty(\Omega)} \in \bochnerex{L^1}{I}{\IR}\). Even if we proceed more rigorously and exploit all the estimates in the proof of Lemma \ref{lemcrazy} more carefully, we arrive at transcendent equations for \(M(\varphi(t))\) which nevertheless suggest exponential growth of the bound on \(\norm{u^\delta_x(t)}{L^\infty(\Omega)}\).  \qedrem
\end{rem}

\begin{rem} \label{remarkcrazy2}
In order to improve the regularity results of Theorem \ref{theorem1}, one were willing to trade regularity in space for regularity in time. In particular, a uniform bound in the space \(\bochner{L^p}{W^{1,q}}\), \(1 < p, q < \infty\) would be sufficient to be able to prove Theorem \ref{theorem1} by means of monotone operators (cf. Remark \ref{remarkmon}). However, the strategy of the proof of Lemma \ref{lemcrazy} can yield an \(W^{1,\infty}(\Omega)\)-bound on \(\sequence{u^\delta}{\delta > 0}\) only. Moreover, the author did not succeed in applying interpolation arguments. \qedrem
\end{rem}

%

Finally, in view of the preceding results we will need to identify the limit of a sequence of functions \(\sequence{f_k}{k \in \IN} \subset \bochnerex{L^p}{\openI}{X}\) that converges weakly in \(\bochnerex{L^p}{\openI}{X}\) as well as weakly in \(X\) for almost all \(t \in (0,T)\). The following lemma will be helpful.

\begin{lem} \label{lemidentification}
Let \(X\) be a Banach space and \(\sequence{f_k}{k \in \IN} \subset \bochnerex{L^p}{\openI}{X}\), \(1 < p < \infty\), a sequence such that
\begin{alignat}{3}
f_k &\weakto f &&\weaktext &&\bochnerex{L^p}{\openI}{X} \label{lemidentconvf} \\
\intertext{and}
f_k(t) &\weakto g(t) &&\weaktext &&X \label{lemidentconvg}
\end{alignat}
for almost all \(t \in (0,T)\) as \(k \to \infty\). Then we may identify \(f(t) = g(t)\) for almost all \(t \in (0,T)\). \qedthm
\end{lem}

\begin{proof}
Let \(\functional \in X'\) and \(\varphi \in C^0([0,T])\). Then we have \(\varphi \, \functional \in \bochnerex{L^q}{\openI}{X'} \cong \bochnerex{L^p}{\openI}{X}'\), \(\frac{1}{p} + \frac{1}{q} = 1\),
and
\[
\inttime{\pairex{\varphi(t) \, \functional}{f_k(t)}{X}} \xrightarrow{\delta \limeps 0} \inttime{\pairex{\varphi(t) \, \functional}{f(t)}{X}}
\]
by \eqref{lemidentconvf}. On the other hand, since
\[
\pairex{\varphi(t) \, \functional}{f_k(t)}{X} \xrightarrow{\delta \limeps 0} \pairex{\varphi(t) \, \functional}{g(t)}{X} 
\]
for almost all \(t \in (0,T)\) by \eqref{lemidentconvg} and
\[
\bigl\lvert \pairex{\varphi(t) \, \functional}{f_k(t)}{X} \bigr\rvert \leq \norm{\varphi}{C^0([0,T])} \norm{\functional}{X'} \norm{f_k(t)}{X} \cma
\]
where \(\norm{f_k}{X}\) is uniformly bounded in \(\bochnerex{L^p}{\openI}{\IR}\) due to \eqref{lemidentconvf}, we may apply Vitali's convergence theorem to obtain
\[
\inttime{\pairex{\varphi(t) \, \functional}{f_k(t)}{X}} \xrightarrow{\delta \limeps 0} \inttime{\pairex{\varphi(t) \, \functional}{g(t)}{X}} \fst
\]
Hence, we conclude
\[
\inttime{\varphi(t) \, \pairex{\functional}{f(t)}{X}} = \inttime{\varphi(t) \, \pairex{\functional}{g(t)}{X}} \fst
\]
Since \(\varphi \in C^0([0,T])\) was arbitrary, it follows that
\[
\pairex{\functional}{f(t)}{X} = \pairex{\functional}{g(t)}{X}
\]
for almost all \(t \in (0,T)\), and the arbitrariness of \(\functional \in X'\) completes the proof.
\end{proof}

\subsection{Proof of Theorem \ref{theorem1}}

Using these results, we are now ready to prove Theorem \ref{theorem1}. Our method is somewhat similar to the approach appearing in the proof of \cite[Theorem 8.2]{gruen2}, even though there are several substantial differences in comparison to our case.

\begin{proof}[Proof of Theorem \ref{theorem1}]
Let \((u^\delta, p^\delta)\) be a solution in the sense of Proposition \ref{proposition3}. In particular, the functions \(u^\delta\) and \(p^\delta\) satisfy the equations \eqref{chap4sysdelta} and the estimates \eqref{chap4energydelta} and \eqref{chap4entropydelta}. Furthermore, note that \(u^\delta\) is continuous and nonnegative on \(\overline{\Omega}_T\).

\beginstep{(1)}{A priori estimates.} As before, conservation of mass is easy to prove employing Proposition \ref{propevtrip}, that is,
\begin{equation} \label{prf4consmass}
\intspace{u^\delta(t)} = \intspace{u_0}
\end{equation}
holds for \(t \in [0,T]\). Hence, we deduce from the energy estimate \eqref{chap4energydelta} and Poincar\'e's inequality that
\begin{equation}
\norm{u^\delta}{\bochner{L^\infty}{W^{1,1}}} \leq C(u_0) \fst \label{prf4boundu}
\end{equation}
Moreover, using \eqref{chap4energydelta} with \(r = 2\), we argue as usual to show that
\begin{equation}
\norm{u^\delta_t}{\timespace} \leq C(u_0,m) \fst \label{prf4boundut}
\end{equation}
Using estimate \eqref{chap4entropydelta}, calculation as in the proof of Proposition \ref{proposition3} (cf. \eqref{prf3estimatep}) gives
\begin{equation}
\norm{p^\delta}{L^2(\Omega_T)} \leq C(u_0) \label{prf4boundp} \fst
\end{equation}
Furthermore, we define the flux
\begin{equation} \label{prf4defJ}
J^\delta = 
\begin{cases}
m(u^\delta) \, p^\delta_x & \text{ on } \possettime{u^\delta} \cma \\
0 & \text{ on } \zeroset{u^\delta}_T \fst
\end{cases}
\end{equation}
By the energy estimate \eqref{chap4energydelta} with \(r = 2\) it holds that
\begin{equation}
\norm{J^\delta}{L^2(\Omega_T)}^2 = \intspacetimeex{\possettime{u^\delta}}{m(u^\delta)^2 \, \abs{p^\delta_x}^2}\leq C(u_0) \label{prf4boundJ} 
\end{equation}

\beginstep{(2)}{The limit \(\delta \to 0\).} From \eqref{prf4boundu}, \eqref{prf4boundut}, \eqref{prf4boundp}, \eqref{prf4boundJ} and Simon's compactness criterion \ref{thmsimon} (note in particular that \(W^{1,1}(\Omega) \cptembedding L^q(\Omega) \embedding H^1(\Omega)'\) holds true for all \(1 \leq q < \infty\) in one space dimension), we deduce the existence of a subsequence, also denoted by \(\sequence{u^\delta}{\delta > 0}\), \(\sequence{p^\delta}{\delta > 0}\) and \(\sequence{J^\delta}{\delta > 0}\), respectively, such that
\begin{alignat}{3}
u^\delta &\weakstarto u &&\weakstartext &&\bochner{L^\infty}{BV} \cma \\
u^\delta &\strongto u &&\strongtext &&\bochnerex{C}{[0,T]}{L^q(\Omega)}, \quad 1 \leq q < \infty \cma \label{prf4strconvusim} \\
u^\delta_t &\weakto u_t &&\weaktext &&\timespace \cma \\
p^\delta &\weakto p &&\weaktext &&\bochner{L^2}{L^2} \cma \\
J^\delta &\weakto J &&\weaktext &&\bochner{L^2}{L^2}
\end{alignat}
as \(\delta \limeps 0\). This proves the regularity results \eqref{thm1regu}--\eqref{thm1regp}. Moreover, since \(u^\delta(0) = u_0\) holds uniformly for \(\delta > 0\) by Proposition \ref{proposition3}, the initial condition \eqref{thm1initial} follows immediately from \eqref{prf4strconvusim}. Similarly, conservation of mass \eqref{thm1consmass} follows from \eqref{prf4consmass} and \eqref{prf4strconvusim} once again. Employing the relation \eqref{prf4defJ} and letting \(\delta \limeps 0\) in \eqref{chap4sysdeltaa}, we obtain
\[
\inttime{\pairh{u_t}{v}} + \intspacetime{J \cdot v_x} = 0
\]
for all \(v \in \bochner{L^2}{H^1}\). The flux \(J\) as well as the pressure \(p\) will be identified below.

\beginstep{(3)}{Identification of \(p\).} From the estimates \eqref{chap4energydelta} and \eqref{chap4entropydelta} together with Lemma \ref{lemunibound} we infer the existence of a further subsequence, again denoted by \(\sequence{u^\delta}{\delta > 0}\), and a set \(S\) satisfying \(\mu(S) = T\) such that
\begin{align}
\intspace{\Q{u^\delta_x(t)}} &\leq C(u_0) \cma \label{prf4bounduxt} \\
\intspace{\dfrac{\abs{u^\delta_{xx}(t)}^2}{\Q{u^\delta_x(t)}^3} + \delta \, \abs{u^\delta_{xx}(t)}^2} &\leq C(u_0, t) \cma  \label{prf4bounduxxt} \\
\intspaceex{\posset{u^\delta(t)}}{m(u^\delta(t))^r \, \abs{p^\delta_x(t)}^2} &\leq C(u_0, r, t) \label{prf4boundtildeJt}
\end{align}
for all \(t \in S\), \(r > 1\), independently of \(\delta > 0\). Now, it follows from \eqref{prf4consmass}, \eqref{prf4bounduxt}, \eqref{prf4bounduxxt} and Lemma \ref{lemcrazy} that
\begin{equation}
\norm{u^\delta(t)}{H^2(\Omega)} \leq C(u_0, t) \label{prf4boundu2t}
\end{equation}
for \(t \in S\) independently of \(\delta > 0\). Indeed, since \eqref{prf4strconvusim} in particular implies
\begin{alignat*}{3}
u^\delta(t) \strongto u(t) &&\strongtext &&L^2(\Omega) 
\end{alignat*}
for every \(t \in S\), we deduce from \eqref{prf4boundu2t} that the whole subsequence weakly converges in \(H^2(\Omega)\), that is
\begin{alignat}{3}
u^\delta(t) &\weakto u(t) && \weaktext &&H^2(\Omega) \label{prf4weakconvu2t}
\end{alignat}
holds for \(t \in S\) as \(\delta \limeps 0\). Since the embedding \(H^2(\Omega) \cptembedding C^{1, \beta}(\overline{\Omega})\) is compact for \(\beta < \frac{1}{2}\), it follows from \eqref{prf4boundu2t} and \eqref{prf4weakconvu2t} once more for \(t \in S\) and the whole subsequence that (cf. \cite[Lemma 8.2]{alt})
\begin{alignat*}{3}
u^\delta(t) &\strongto u(t) &&\strongtext &&C^{1, \beta}(\overline{\Omega}) \fst
\end{alignat*}
This in particular implies that
\begin{alignat}{2}
u^\delta(t, .) &\strongto u(t, .) &&\text{ uniformly on }\overline{\Omega} \cma \label{prf4uniconvu} \\
u^\delta_x(t, .) &\strongto u_x(t, .) &&\text{ uniformly on }\overline{\Omega} \label{prf4uniconvux}
\end{alignat}
for \(t \in S\). Now, \eqref{thm1regu2} follows obviously from \eqref{prf4weakconvu2t}. Furthermore, from the uniform convergence \eqref{prf4uniconvu} and the fact that \(u^\delta(t) \geq 0\) for \(t \in S\) and \(\delta > 0\) by Proposition \ref{proposition3pos} (i), we deduce that indeed \(u(t) \geq 0\) holds, which is \eqref{thm1nonneg}. Finally, since \(u^\delta(t) \in H^1_0(\Omega)\) for almost all \(t \in (0,T)\), the boundary regularity \eqref{thm1boundary} follows from \eqref{prf4weakconvu2t} after possibly restricting \(S\) by a set of measure zero.

To identify the pressure \(p\), we argue as usual and use the relation
 \[
- \ddx \Qex{u^\delta_x(t)} = - \frac{u^\delta_{xx}(t)}{\Q{u^\delta_x(t)}^3}
 \]
and the fact
\[
\frac{1}{\Q{u^\delta_x(t)}^3} \leq 1
\]
to infer from the pointwise convergence of \(u^\delta_x(t)\) \eqref{prf4uniconvux}, the weak convergence of \(u^\delta_{xx}(t)\) \eqref{prf4weakconvu2t} and Vitali's convergence theorem that
\begin{alignat*}{3}
- \ddx \Qex{u^\delta_x(t)} &\weakto - \ddx \Qex{u_x(t)} &&\weaktext &&L^2(\Omega) \fst
\end{alignat*}
Moreover, by \eqref{prf4bounduxxt}, we have for arbitrary \(v \in L^2(\Omega)\)
\[
\delta \intspace{u^\delta_{xx}(t) \cdot v} \leq \delta \, \norm{u^\delta_{xx}(t)}{L^2(\Omega)} \norm{v}{L^2(\Omega)} \xrightarrow{\delta \limeps 0} 0
\]
for \(t \in S\). Hence,
\begin{equation} \label{prf4weakconvpt}
p^\delta(t) \weakto - \ddx \Qex{u_x(t)} \weaktext L^2(\Omega)
\end{equation}
for all \(t \in S\) as \(\delta \limeps 0\). This together with Lemma \ref{lemidentification} proves \eqref{thm1eqb}.


We now fix \(t \in S\) and intend to show that
\begin{equation} \label{prf4weakconvpxtloc}
p^\delta_x(t) \weakto p_x(t) \weaktext L^2(U) \qquad \forall \, U \subset \subset \posset{u(t)}
\end{equation}
by similar arguments as in the proof of Proposition \ref{proposition3}. First of all, note that the function \(u(t,.) : \overline{\Omega} \to \IR\) is continuous and nonnegative on \(\overline{\Omega}\) by the uniform convergence \eqref{prf4uniconvu}. We choose an arbitrary \(\possubset \subset \subset \posset{u(t)}\) and denote
\[
\posparam := \min_{x \in \overline{\possubset}} u(t,x) \fst
\]
Owing to \eqref{prf4uniconvu}, we have \(u^\delta(t, .) \geq \frac{\posparam}{2}\) on \(\possubset\) for \(\delta > 0\) sufficiently small, and in particular \(\possubset \subset \posset{u^\delta(t)}\). Therefore, from the energy estimate \eqref{prf4boundtildeJt} with \(r = 2\) we deduce
\[
c_{\mobextra}^2 \left( \frac{\posparam}{2} \right)^{2n} \intspaceex{\possubset}{\abs{p^\delta_x(t)}^2} \leq \intspaceex{\posset{u^\delta(t)}}{m(u^\delta)^2 \, \abs{p^\delta_x(t)}^2} \leq C(u_0, t)
\]
for \(\delta > 0\) sufficiently small. This implies
\[
\norm{p^\delta_x(t)}{L^2(\possubset)} \leq C(u_0, c_{\mobextra}, t, \possubset) \fst
\]
For the weak limit \(q \in L^2(\possubset)\) of a subsequence of \(\sequence{p^\delta_x(t)}{\delta > 0}\), we infer from the weak convergence of \(p^\delta(t)\) \eqref{prf4weakconvpt} that for arbitrary \(\varphi \in C^\infty_0(\possubset)\) the relation
\[
\intspaceex{\possubset}{q \cdot \varphi} = - \intspaceex{\possubset}{p(t) \cdot \varphi_x}
\]
is satisfied. Hence, we may identify \(q = p_x(t)\) and conclude that the whole sequence converges for \(t \in S\). The arbitrariness of \(\possubset\) proves \eqref{prf4weakconvpxtloc} and in turn \eqref{thm1regpx}.

\beginstep{(4)}{Identification of \(J\).} It remains to identify \(J\). To this end, we shall prove that
\begin{equation} \label{prf4convmint}
\intspace{J^\delta(t) \cdot w} = \intspaceex{\posset{u^\delta(t)}}{m(u^\delta(t)) \, p^\delta_x(t) \cdot w} \xrightarrow{\delta \limeps 0} \intspaceex{\posset{u(t)}}{m(u(t)) \, p_x(t) \cdot w}
\end{equation}
for \(w \in L^2(\Omega)\) and \(t \in S\). For arbitrary \(\rho > 0\) we split
\[
\begin{split}
\intspaceex{\posset{u^\delta(t)}}{m(u^\delta(t)) \, p^\delta_x(t) \cdot \varphi}
&= \intspaceex{\substack{\setex{u(t) > \rho} \\ \cap \, \setex{u^\delta(t) > 0}}}{m(u^\delta(t)) \, p^\delta_x(t) \cdot \varphi} \\
& \quad + \intspaceex{\substack{\setex{u(t) \leq \rho} \\ \cap \, \setex{u^\delta(t) > 0}}}{m(u^\delta(t)) \, p^\delta_x(t) \cdot \varphi} \\
& =: \mathcal{I}_1 + \mathcal{I}_2 \fst
\end{split}
\]
Regarding \(\mathcal{I}_1\), note that \(\setex{u(t) > \rho} \cap \setex{u^\delta(t) > 0} = \setex{u(t) > \rho}\) for \(\delta >0\) sufficiently small due to the uniform convergence of \(u^\delta\) \eqref{prf4uniconvu}. As before, we deduce from \eqref{prf4uniconvu} and \eqref{prf4weakconvpxtloc} that
\[
\mathcal{I}_1 \xrightarrow{\delta \limeps 0} \intspaceex{\possetex{u(t)}{\rho}}{m(u(t)) \, p_x(t) \cdot w} \fst
\]
On the other hand, writing \(E_\rho := \setex{u(t) \leq \rho} \cap \setex{u^\delta(t) > 0}\), we infer from \eqref{prf4boundtildeJt} with \(r = \frac{3}{2}\) (cf. Remark \ref{remarkenergyr} at this point) that
\[
\begin{split}
\abs{\mathcal{I}_2} & \leq \absB{\intspaceex{E_\rho}{m(u^\delta(t)) \, p^\delta_x(t) \cdot w}} \\
& \leq \intspaceex{E_\rho}{\abs{m(u^\delta(t))^\frac{1}{4} \, w \cdot m(u^\delta(t))^\frac{3}{4} \, p^\delta_x(t)}} \\
& \leq \Bigl( \intspaceex{E_\rho}{m(u^\delta(t))^\frac{1}{2} \, w^2} \Bigr)^\frac{1}{2} \Bigl( \intspaceex{E_\rho}{m(u^\delta(t))^\frac{3}{2} \, \abs{p^\delta_x(t)}^2} \Bigr)^\frac{1}{2} \\
& \leq \rho^\frac{n}{4} \cdot C(m_0) \, \norml{w} \, \intspaceex{\setex{u^\delta(t) > 0}}{m(u^\delta(t))^\frac{3}{2} \, \abs{p^\delta_x(t)}^2} \Bigr)^\frac{1}{2} \\
& \leq C(u_0, m_0, t, w) \, \rho^\frac{n}{4} \fst
\end{split}
\]
for any \(w \in L^2(\Omega)\) and independently of \(\delta > 0\). Thus, the arbitrariness of \(\rho > 0\) proves \eqref{prf4convmint}. In particular, using Lemma \ref{lemidentification} once more, we finally infer that
\[
J(t) =
\begin{cases}
m(u(t)) \, p_x(t) & \text{ on \(\posset{u(t)}\), }\\
0 & \text{ on \(\zeroset{u(t)}\). }\\
\end{cases}
\]
for almost all \(t \in (0,T)\). This completes the proof of Theorem \ref{theorem1}.
\end{proof}

\subsection{Proof of Theorem \ref{theorem2}}

In order to prove Theorem \ref{theorem2} (i), we will basically repeat the arguments from the proof of Proposition \ref{proposition3pos} (ii), using the uniform convergence of the functions \(u^\delta(t, .)\) for \(t \in S\) and estimate \eqref{chap4boundGdelta}. Regarding Theorem \ref{theorem2} (ii), we adopt once more the arguments of Bernis and Friedman (cf. \cite[Corollary 4.5]{berfri}), who have proved the very same result in the setting of the standard thin-film equation.

\begin{proof}[Proof of Theorem \ref{theorem2}]
Let \(u\) be a solution in the sense of Theorem \ref{theorem1} and \(\sequence{u^\delta}{\delta > 0}\) be the sequence of functions that we have obtained in the course of the previous proof. In particular, we may suppose that
\begin{alignat}{2}
u^\delta(t, .) &\strongto u(t, .) &&\text{ uniformly on }\overline{\Omega} \label{prf5uniconv}
\end{alignat}
as \(\delta \to 0\) for \(t \in S\), \(\mu(S) = T\), and \(\beta < \frac{1}{2}\), where the functions \(\sequence{u^\delta}{\delta > 0}\) are solutions in the sense of Proposition \ref{proposition3}

\beginstep{(1)}{ad (i).} Suppose that \(n \geq 2\). Then the conditions of Proposition \ref{proposition3pos} are satisfied such that \(u^\delta(t, .)\) is nonnegative on \(\overline{\Omega}\) and
\begin{equation}
\intspace{G(u^\delta(t,x))} \leq C(u_0) \label{prf5boundG}
\end{equation}
holds for all \(t \in S\) and independently of \(\delta > 0\). As in the proof of Proposition \ref{proposition3pos}, we argue by contradiction and assume that there exists \(t_0 \in S\) such that \(\mu(E) > 0\), where \(E := \zeroset{u(t_0)} \subseteq \overline{\Omega}\).
Let \(\rho > 0\). Since \(u^\delta(t_0, .) \to u(t_0, .)\) uniformly on \(\overline{\Omega}\) by \eqref{prf5uniconv},
we have
\[
u^\delta(t_0, x) < \rho
\]
for \(x \in E\) and \(\delta > 0\) sufficiently small.
Moreover, it follows immediately from the monotonicity of the functions \(G_\epsilon(.)\) for \(\epsilon > 0\) (see \eqref{chap3gmon}) that the limit \(G(s) = \lim_{\epsilon \to 0} G_\epsilon(s)\) is monotonically decreasing for \(0 < s \leq a\). Hence, remembering that \(u^\delta(t_0,x) > 0\) for almost all \(x \in \Omega\) if \(n \geq 2\) and using the the growth properties of \(G\) derived in \eqref{chap3growthG}, we infer that
\[
G(u^\delta(t_0, x)) \geq G(\rho) \geq
\begin{cases}
c \, \log \frac{1}{\rho} &\text{ if } n = 2 \cma \\
c \, \rho^{2-n} &\text{ if } n > 2 \cma
\end{cases}
\]
for almost all \(x \in E\), a positive constant \(c > 0\) and \(\rho > 0\) sufficiently small. Therefore,
\[
\liminf\limits_{\delta \limeps 0} \intspace{G(u^\delta(t_0, x))} \geq
\begin{cases}
c \, \log \frac{1}{\rho} \, \mu(E) &\text{ if } n = 2 \cma \\
c \, \rho^{2-n} \, \mu(E) &\text{ if } n > 2 \fst
\end{cases}
\]
Letting \(\rho \limeps 0\), we obtain a contradiction to \eqref{prf5boundG} and the assertion \(\mu \left( \zeroset{u(t_0)} \right) = 0\) follows.

\beginstep{(2)}{ad (ii).} Using the result above together with the uniform convergence of \(u^\delta(t, .)\) for \(t \in S\), we may argue as in the proof of Proposition \ref{proposition3pos} and employ Fatou's lemma to show that
\begin{equation}
\intspace{G(u(t, x))} \leq C(u_0) \fst \label{prf5boundG2}
\end{equation}
Now, assume that \(u(t_0, x_0) = 0\) for any \((t_0, x_0) \in S \times \overline{\Omega}\). Since
\begin{align}
u(t_0, .) &\in H^2(\Omega) \embedding C^{1, \frac{1}{2}}(\overline{\Omega}) \cma \label{prf5holder} \\
u(t_0, .) &\geq 0 \text{ on } \overline{\Omega} \notag
\end{align}
hold, we infer at first for \(x_0 \in \mathring{\Omega}\) that
\begin{equation} \label{prf5uux0}
u(t_0, x_0) = u_x(t_0, x_0) = 0 \fst
\end{equation}
If \(x_0 \in \partial \Omega\), \eqref{prf5uux0} follows from the boundary regularity \(u_x(t_0, .) = 0\) on \(\partial \Omega\) (cf. \eqref{thm1boundary}). Hence, we deduce from \eqref{prf5holder} and \eqref{prf5uux0} that
\[
u(t_0, x) \leq C(t_0) \, \abs{x - x_0}^{\frac{3}{2}}
\]
for arbitrary \(x \in \overline{\Omega}\). This leads to
\[
\intspace{\abs{x-x_0}^{\frac{3(2-n)}{2}}} \leq C(t_0) \intspace{u(t_0, x)^{2-n}} \cma
\]
where the left-hand side is equal to \(+ \infty\) if \(\frac{3(2-n)}{2} \leq -1\), that is if \(n \geq \frac{8}{3}\). On the other hand, in view of the growth of \(G\) (cf. again \eqref{chap3growthG}) and \eqref{prf5boundG2}, the right-hand side is bounded by
\[
\intspace{u(t_0, x)^{2-n}} \leq C \intspace{G(u(t_0, x))} \leq C(u_0) \fst
\]
Hence, we obtain a contradiction and the proof of Theorem \ref{theorem2} is complete.
\end{proof}

\clearsection

\section{Perspectives} \label{secperspectives}

\subsection{Uniqueness}

Regarding standard thin-film equations, there is no general answer to the question of uniqueness so far. Bernis and Friedman \cite{berfri} proved that solutions to the one-dimensional thin-film equation are unique, provided that the initial data are positive and \(n \geq 4\) holds for the growth exponent \(n\) of the mobility \(m(.)\) at zero. The proof relies on the continuity of \(u\) and the fact that initially positive solutions stay positive on \(\overline{\Omega}_T\) if \(n \geq 4\), which has also been proved in \cite{berfri}. In this context, we refer to Beretta, Bertsch and Dal Passo \cite{berberdal} for an example of non-uniqueness in one space dimension.

As we end up with significantly less regularity results for \(u\) in Theorem \ref{theorem1}, we can not hope to prove uniqueness in our case by the arguments appearing in \cite{berfri}. In particular, the solution \(u\) is not necessarily continuous and positive on \(\overline{\Omega}_T\) for arbitrary \(n \geq 1\).


\subsection{Additional Singular Terms}

Following the derivation of Oron, Davis and Bankoff \cite{oron}, Gr\"{u}n and Rumpf \cite{gruenrumpf} and Gr\"{u}n \cite{gruen2} considered thin-film flows of the form
\[
u_t + \divergence \left( m(u) \, \nabla (\triangle u - W'(u)) \right) = 0 \qquad \text{ in } (0,T) \times \Omega \subset \IR^{d+1} \cma
\]
where the generalized pressure \(p = - \triangle u + W'(u)\) includes a nonlinear term \(W'(u)\) that is supposed to model additional effects arising for instance from molecular interactions (e.g. van der Waals forces) or gravity. In typical applications, the function \(W'\) will be singular at zero.
One has in mind energy functions \(W\) of the form
\[
W(s) = - \tfrac{H_1}{1-\rho_1} s^{1-\rho_1} + \tfrac{H_2}{1-\rho_2} s^{1-\rho_2}
\]
for some positive constants \(H_1, H_2 > 0\) and real numbers \(\rho_1 > \rho_2 > 1\). For example, the classical 6--12 Lennard--Jones--Potential (which models attractive van der Waals forces at long ranges together with repulsive forces caused by Pauli repulsion at short ranges) entails \(\rho_1 = 9\) and \(\rho_2 = 3\). To be more general, we suppose that \(W\) can be decomposed into
\[
W(s) = W_+(s) + W_-(s)
\]
with a convex, nonnegative function \(W_+\) and a concave function \(W_-\), both satisfying certain growth conditions. For the precise assumptions on \(W'\) as well as further information on the physical background, we refer the reader to \cite{gruenrumpf} and references therein.

In our case, we would hence consider the generalized pressure
\begin{equation} \label{chap5singularp}
p = - \ddx \Qex{u_x} + W'(u)
\end{equation}
and formally derive the energy estimate
\begin{equation} \label{chap5singularenergy}
\begin{split}
\sup\limits_{t \in [0,T]} \intspace{\Q{u_x(t)}} + &\intspace{W(u(t))} + \intspacetime{ m(u) \, \abs{p_x}^2 } \\
&\leq \intspace{\Q{\smash[b]{u_{0,x}}}} + \intspace{W(u_0)} \cma
\end{split}
\end{equation}
where we may assume \(W\) to be nonnegative. Regarding the entropy estimate, one ends up with
\begin{equation} \label{chap5singularentropy}
\begin{split}
\sup\limits_{t \in [0,T]} \intspace{G(u(t))} + \intspacetime{\dfrac{\abs{u_{xx}}^2}{\Q{u_x}^3}} &+ \intspacetime{W''_+(u) \abs{u_x}^2}  \\
\leq \intspace{G(u_0)} &- \intspacetime{W''_-(u) \abs{u_x}^2} \fst
\end{split}
\end{equation}
In addition, under some reasonable assumptions on \(W\) we may suppose that the second term on the right-hand side is bounded by
\begin{equation} \label{chap5singularentropyr}
- \intspacetime{W''_-(u) \abs{u_x}^2} \leq C(W) \intspacetime{\abs{u_x}^2} \cma
\end{equation}

In the setting of the standard thin-film equation, it is now easy to estimate the right-hand side of \eqref{chap5singularentropyr} uniformly due to the boundedness of \(\nabla u\) provided by the energy estimate. In our case, we do not have any comparable uniform estimates at our disposal. Also, the author did not succeed in absorbing the second term on the right-hand side of \eqref{chap5singularentropy}. Since our overall strategy relies on a valid entropy estimate, we may not expect an existence result in the case of a generalized pressure having the form \eqref{chap5singularp} with a general, possibly singular function \(W'\).

However, observe that if we require the corresponding energy function \(W:\IR^+ \to \IR \) to be purely convex, the second term on the right-hand side of \eqref{chap5singularentropy} vanishes and we would indeed be able to prove an existence result in the spirit of Theorem \ref{theorem1}. In addition, note that the quality of our a priori estimates for \(u\) significantly improves on account of the entropy estimate \eqref{chap5singularentropy} for strictly convex energy functions \(W\), as we would then be able to estimate
\[
\norm{u_x}{L^2(\Omega_T)} \leq C(u_0) \fst
\]
Moreover, if \(W\) is singular at zero, the energy estimate \eqref{chap5singularenergy} yields additional nonnegativity and positivity results.

Let us mention that in order to proceed rigorously one would have to introduce a further regularization by -- for instance -- linearizing \(W'\) at zero in a first step. For the merely technical details of this procedure, we refer the reader to \cite{gruenrumpf} once again.

Finally, we want to draw the reader's attention to the thesis of Becker \cite{becker} once more, where the author considers the thin-film equation comprising the generalized pressure with the nonlinear surface tension term and a singular term \(W'\) as given by \eqref{chap5singularp}. In particular, Becker successfully proved an existence result for a discrete solution to this equation in multiple space dimensions.

\subsection{Higher Space Dimensions} \label{sechigher}

In the setting of the standard thin-film equation in higher space dimensions, which reads as
\[
u_t + \divergence \left( m(u) \, \nabla \triangle u \right) = 0
\qquad \text{ in } (0,T) \times \Omega \subset \IR^{d+1}
\]
in its most simplified form, existence of a weak solution has been proved by Gr\"{u}n \cite{gruen1} among others. In comparison to the results of Bernis and Friedman \cite{berfri} for the one-dimensional case, the regularity of the solution is subject to some essential limitations. In particular, one may not expect the solutions to be bounded and globally H\"{o}lder-continuous, as both results rely on one-dimensional Sobolev embeddings. As a consequence, existence of solutions in the case of general unbounded mobilities \(m\)
is in fact an open question up to now in the higher-dimensional setting.

Considering system \eqref{P1} in higher space dimensions, that is
\begin{equation} \label{chap5highersys}
\begin{aligned}
u_t &- \divergence \left( m(u) \, \nabla p \right) = 0 \\
p &= - \divergence \Qex{\nabla u} 
\end{aligned}
\qquad \text{ in } (0,T) \times \Omega \subset \IR^{d+1} \cma
\end{equation}
we shall not expect any better outcome and at first confine ourselves to bounded mobilities \(m\) having potential growth at zero. Moreover, an intermediate H\"{o}lder-continuity result as in Corollary \ref{corcont} is also not to be expected. However, since we have used Corollary \ref{corcont} as a handy tool on an approximative level only and were not able to conserve global continuity in the final result anyway, there might be a workaround.

More importantly, note that we took advantage of the identity
\[
p = - \ddx \Qex{u_x} = -\frac{u_{xx}}{\Q{u_x}^3}
\]
several times, which holds true in one space dimension. In higher space dimensions, this relation becomes far more complicated. One calculates
\begin{equation}
p = - \divergence \Qex{\nabla u} = - \frac{\triangle u + \abs{\nabla u}^2 \triangle u - \prodex{\nabla u}{D^2 u \, \nabla u}{\IR^d}}{\Q{\nabla u}^3} \label{chap5higherp} \fst
\end{equation}

Up to a certain level, our approach might nevertheless be applicable in multiple space dimensions. Regularizing equations \eqref{chap5highersys} in the fashion above by introducing the positive mobility
\[
m_\epsilon(s) := m(s) + \epsilon \qquad \forall \, s \in \IR
\]
and the coercive nonlinear operator \(A_\delta : H^1(\Omega) \to H^1(\Omega)'\) given by
\[
\pairh{A_\delta(u)}{v} := \intspace{\leftA \Qex{\nabla u} + \delta \nabla u \rightA \cdot \nabla v } \qquad \forall \, u, v \in H^1(\Omega) \cma
\]
the existence result for a discrete solution as in Lemma \ref{lemdisc} remains valid without any limitations, and we obtain the energy estimate
\[
\sup\limits_{t \in [0,T]} \intspace{\Q{\nabla u(t)}} + \frac{\delta}{2} \norml{\nabla u(t)}^2 + \intspacetime{ m_{\epsilon}(u) \, \abs{\nabla p}^2 } \leq C(\approxu) \fst
\]
Concerning the proof of Proposition \ref{proposition1}, one would avoid making use of relation \eqref{chap5higherp} and instead argue by the monotonicity of \(A_\delta\) to identify
\[
p = A_\delta(u) \text{ in } \bochnerex{L^2}{I}{H^1(\Omega)'}
\]
in the limit (cf. Remark \ref{remarkmon}). Eventually, supposing that \(\partial \Omega\) is sufficiently smooth, we may employ regularity results for second-order nonlinear elliptic equations afterwards to find that \(u \in \bochner{L^2}{H^2}\) (cf. Remark \ref{remarkreg}).

However, the prospects decrease when we try to derive the entropy estimate. This would require multiplying the right-hand side of \eqref{chap5higherp} by \(- \triangle u\), but it is not obvious whether we arrive at any reasonable estimate in terms of \(u\) in the end. Moreover, even if the derivation of an analogous estimate would be possible, there is no hope to prove a result in the spirit of Lemma \ref{lemcrazy}, the proof of which heavily relies on an one-dimensional embedding result with critical exponents. In other words, it is open how to derive any reasonable a priori estimate for \(\nabla u\) in order to identify the pressure \(p\) in the limit.

\clearsection

\newpage
\appendix

\section{Appendix}

\begin{thmapx}[Peano's existence theorem, {\cite[Theorem 3.B]{zeidler1}}] \label{thmpeano}
Let \(N \geq 1\), \(t_0 \in \IR\), \(\mathbf{y_0} \in \IR^N\) and \(\mathbf{f}:\IR \times \IR^N \to \IR^N\) be given. Consider the initial value problem
\begin{equation}\label{apx1ivp}
\begin{split}
\frac{d}{dt}\mathbf{x}(t) &= \mathbf{f}(t, \mathbf{x}(t)) \cma \\
\mathbf{x}(t_0) &= \mathbf{y_0} \fst
\end{split}
\end{equation}
For \(a, b > 0\) we set
\[
Q := \bigl\lbrace (t, \mathbf{y}) \in \IR \times \IR^N : \abs{t - t_0} \leq a, \norm{\mathbf{y} - \mathbf{y_0}}{} \leq b \bigr\rbrace \fst
\]
If \(f\) is continuous on \(Q\) and bounded by
\[
\norm{\mathbf{f}(t,\mathbf{y})}{} \leq K \quad \forall (t,\mathbf{y}) \in Q \cma
\]
then there exists a continuously differentiable solution to the initial value problem \eqref{apx1ivp} on the interval \([t_0 - \tau, t_0 + \tau]\), where
\[
\tau := \min \Bigl\lbrace a, \dfrac{b}{K} \Bigr\rbrace \fst
\]
\qedthm
\end{thmapx}


%
\begin{prop}[Evolution triples, {\cite[Chapter 23]{zeidler2}}] \label{propevtrip}
Let \(V\) be a real, separable, reflexive Banach space, \(H\) be a real, separable Hilbert space, and the embedding \(V \embedding H\) be continuous. Then the triple
\[
V \embedding H \embedding V'
\]
is called an ``evolution triple''. Moreover, we define for \(1 < p, q < \infty\), \(\frac{1}{p} + \frac{1}{q} = 1\),
\[
W^1_p(\openI; V, H) := \lbrace u \in \bochnerex{L^p}{\openI}{V} : u_t \in \bochnerex{L^q}{\openI}{V'} \rbrace \fst
\]
The following hold:
\begin{enumerate}
\item[(i)] The space \(W^1_p(\openI; V, H)\) is a real Banach space.
\item[(ii)] The embedding
\[
W^1_p(\openI; V, H) \embedding \bochnerex{C}{[0,T]}{H}
\]
is continuous.
\item[(iii)] The space \(\bochnerex{C^1}{[0,T]}{V}\) is dense in \(W^1_p(\openI; V, H)\).
\item[(iv)] For arbitrary \(u, v \in W^1_p(\openI; V, H)\) and \(0 \leq t_1, t_2 \leq T\), the generalized integration by parts formula
\[
\prodex{(u(t_2)}{v(t_2)}{H} - \prodex{(u(t_1)}{v(t_1)}{H} = \int\limits_{t_1}^{t_2} {\pairex{u'(t)}{v(t)}{V} - \pairex{v'(t)}{u(t)}{V} \, dt}
\]
holds. \qedthm
\end{enumerate}
\end{prop}

\begin{thmapx}[Simon, {\cite[Corollary 4]{simon}}] \label{thmsimon}
Let \(X\), \(B\) and \(Y\) be Banach spaces such that
\[
X \cptembedding B \embedding Y \fst
\]
For a set of functions \(F \subset \bochnerex{L^1}{\openI}{X}\) we denote by
\[
\partial_t F := \lbrace \partial_t f : f \in F \rbrace
\]
the set of the distributional derivatives of functions in \(F\). It holds:
\begin{enumerate}
\item[(i)] Let \(F\) be bounded in \(\bochnerex{L^p}{\openI}{X}\) where \(1 \leq p < \infty\), and \(\partial_t F\) be bounded in \(\bochnerex{L^1}{\openI}{Y}\). Then \(F\) is relatively compact in \(\bochnerex{L^p}{\openI}{B}\).
\item[(ii)] Let \(F\) be bounded in \(\bochnerex{L^\infty}{\openI}{X}\) and \(\partial_t F\) be bounded in \(\bochner{L^r}{\openI}{Y}\) where \(r > 1\). Then \(F\) is relatively compact in \(\bochnerex{C}{[0,T]}{B}\). \qedthm
\end{enumerate}
\end{thmapx}

\clearsection

\bibliography{da-en}

\newpage
\cleardoublepage
\thispagestyle{plain}

\section*{Selbstst\"{a}ndigkeitserkl\"{a}rung}

Hiermit erkl\"{a}re ich, dass ich die vorliegende Arbeit selbstst\"{a}ndig und
nur unter Verwendung der angegebenen Quellen und Hilfsmittel angefertigt
habe.

\vspace*{1cm}

\begin{center}
Erlangen, den 8. Juni 2009
\hspace*{2.5cm}
\begin{tabular}{@{}p{5cm}@{}}
\phantom{Jan Friederich}\\
\hline\centering\footnotesize{(Jan Friederich)}
\end{tabular}
\end{center}

\end{document}